\documentclass[11pt,oneside,reqno]{amsart}

\calclayout

\usepackage[utf8]{inputenc}
\usepackage[T1]{fontenc}
\usepackage[english]{babel}
\usepackage[dvipsnames]{xcolor}
\usepackage{amssymb , amsmath, amsthm}
\usepackage{mathtools}
\mathtoolsset{showonlyrefs}
\usepackage{hyperref}
\usepackage{appendix}

\allowdisplaybreaks

\newcommand{\R}{\mathbb{R}}
\newcommand{\N}{\mathbb{N}}

\newcommand{\eps}{\varepsilon}
\newcommand{\dm}{\mathrm{dim}}
\newcommand{\supp}{\mathrm{supp}}
\renewcommand{\geq}{\ge}
\renewcommand{\leq}{\le}

\newtheorem{theorem}{Theorem}[section]
\newtheorem{coro}[theorem]{Corollary}
\newtheorem{proposition}[theorem]{Proposition}
\newtheorem{definition}[theorem]{Definition}
\newtheorem{lemma}[theorem]{Lemma}

\newtheorem{rem}[theorem]{Remark}

\newcommand{\ukm}{u_k^{(m)}}
\newcommand{\utkm}{\tilde{u}_k^{(m)}}
    \newcommand{\sikm}{\sigma_k^{(m)}}
    \newcommand{\phkm}{\phi_k^{(m)}}
\newcommand{\de}{\, \mathrm{d}}
\newcommand{\del}{\partial}

\DeclareMathOperator{\Vol}{Vol}
\newcommand{\wk}{\rightharpoonup}
\newcommand{\wks}{\xrightharpoonup{*}}

\DeclareMathOperator{\dist}{dist}

\DeclareMathOperator{\spn}{span}

\newcommand{\dV}{\de V_\Omega}
\newcommand{\dA}{\de A_{\del \Omega}}

\renewcommand{\div}{\operatorname{div}}

\newcommand{\abs}[1]{\left\lvert #1 \right\rvert}
\newcommand{\set}[1]{\left\{ #1 \right\}}

\newcommand{\bigo}[2][]{O_{#1}\left( #2 \right)}
\newcommand{\smallo}[2][]{o_{#1}\left( #2 \right)}

\renewcommand{\le}{\leqslant}
\renewcommand{\ge}{\geqslant}
\renewcommand{\hat}{\widehat}
\renewcommand{\tilde}{\widetilde}
\renewcommand{\epsilon}{\varepsilon}
\renewcommand{\phi}{\varphi}

\newcommand{\CR}{\mathcal{R}}

\newcommand{\BS}{\mathbf{S}}
\newcommand{\BT}{\mathbf{T}}

\newcommand{\bo}\boldsymbol{}

\newcommand{\RC}{\mathrm{C}}

\newcommand{\RL}{\mathrm{L}}

\newcommand{\RW}{\mathrm{W}}

\title[Sharp upper bounds for Steklov eigenvalues]{Sharp upper bounds for Steklov eigenvalues on \\ warped products 
}
\author[Brisson]{Jade Brisson}
\address{D\'epartement de math\'ematiques et de statistique, Universit\'e de Moncton, Pavillon R\'emi-Rossignol, Nouveau-Brunswick, E1A 7R1, Canada
}
\email{jade.brisson@umoncton.ca}

\author[Colbois]{Bruno Colbois}
\address{Institut de Math\'ematiques, Universit\'e de Neuch\^atel, Rue Emile-Argand 11, 2000 Neuch\^atel, Suisse}
\email{bruno.colbois@unine.ch}

\author[Girouard]{Alexandre Girouard}\address{{D\'epartement de math\'ematiques et de statistique}, {Pavillon Alexandre-Vachon}, {Universit\'e Laval}, {Qu\'ebec}, {G1V 0A6}, {Canada}}
\email{alexandre.girouard@mat.ulaval.ca}

\author[Gittins]{Katie Gittins}
\address{Department of Mathematical Sciences, Durham University, Mathematical Sciences and Computer Science Building, Upper Mountjoy Campus, Stockton Road, Durham, DH1 3LE, United Kingdom.}
\email{katie.gittins@durham.ac.uk}

\author[Lagac\'e-Taylor]{Jean Lagac\'e-Taylor}
\address{Department of Mathematics, King's College London, The Strand, London, WC2R 2LS, United Kingdom}
\email{jean.lagace-taylor@kcl.ac.uk}


\date{\today}

\begin{document}
\begin{abstract}
  We obtain sharp isoperimetric-type upper bounds for Steklov eigenvalues on warped products whose bases are compact Riemannian manifolds with boundary and whose fibres are closed Riemannian manifolds.
  These bounds involve the volume of the warped product and of its boundary, as well as the eigenvalues of the Laplace operator on the fibre  and the Steklov eigenvalues of the base. 
  After separation of variables, these bounds are obtained by comparison with auxiliary Helmholtz–Steklov problems on the base involving weighted Laplace operators naturally induced by the warping function.
  We also study the effect of various natural geometric constraints, such as 
  the warping function being identically equal to one on the boundary. The proofs
  make use of various spectral approximation methods, including homogenisation on 
  manifolds when the fibres are one-dimensional.
\end{abstract}

\maketitle

\section{Introduction}
 Let $(\Omega,g_\Omega)$ be a compact, connected Riemannian manifold of dimension
$d \ge 1$, with smooth boundary $\del \Omega$ and let $(\Sigma, g_\Sigma)$ be a
closed, connected Riemannian manifold of dimension $n \ge 1$. Given a smooth
positive function $0 < h \in \RC^\infty(\Omega)$, consider the warped product
$M_h := \Omega \times_h \Sigma$ of dimension $N = d + n$. That is, the product manifold $M = \Omega
\times \Sigma$ is equipped with the Riemannian metric $g_h$ defined by
\begin{equation}
    g_h(x,q) := g_{\Omega}(x) + h^2(x) g_\Sigma(q).
\end{equation}
On $M_h$, we consider the Steklov eigenvalue problem, which consists in finding
all $\sigma \in \R$ and nonzero $u \in \RC^\infty(M)$ that satisfy
\begin{equation}
    \label{eqdef:Steklov}
    \begin{cases}
        \Delta u = 0 & \text{in } M_h, \\
        \del_\nu u = \sigma u & \text{on } \del M_h.
    \end{cases}
\end{equation}
Here, $\Delta = \Delta_{g_h}$ is the Laplacian associated with the warped
metric $g_h$, and $\del_\nu u$ is the outward-pointing normal derivative of the
function $u$. The Steklov eigenvalues form an unbounded, non-decreasing sequence,
\begin{equation}
    0 = \sigma_0(M_h) < \sigma_1(M_h) \le \sigma_2(M_h) \le \dotso \nearrow
    \infty.
\end{equation}
See~\cite{LeMaPo2023} for an introduction and \cite{CoGiGoSh2024,GiPo2017} for surveys of the recent literature on the spectral geometry of the Steklov problem.

In order to obtain upper bounds for the Steklov eigenvalues of Riemannian
manifolds, it is necessary to impose geometric constraints (see \cite{CiGi2018, CoElGi2019, FrSc2011}). In \cite{BrCoGi2024, CoGiGi2019,CoVe2021,Sel25,Tsc24},
upper bounds for Steklov eigenvalues are obtained for manifolds with metrics of
revolution type. This class is too restrictive for the strongest approximation
results that we want to prove. However, metrics of revolution are warped metrics in
the special case where $\Omega$ is an interval, and where $h$ may be allowed to
vanish at one or both ends of that interval. It is therefore natural to seek
upper bounds and approximation results in the class of warped products. Various
upper bounds under restrictions on the class of warped metrics can be found in~\cite{DaHeNi2021,DaKaNi2021}  where the base is an interval, \cite{Xi2021,Xi2022} for manifolds diffeomorphic to a ball, and \cite{GiPoly2024} for general warped products.

\subsection{Main results}  
Our first result is an isoperimetric-type upper bound involving the eigenvalues of the Laplace operator on the fibre $\Sigma$. We list them as
$$0=\lambda_0<\lambda_1\leq\lambda_2\leq\cdots\to+\infty.$$
\begin{theorem}\label{thm:isoptype}
    Let $M_h=\Omega\times_h\Sigma$ be a warped product of dimension $N \geq 3$ with fibre $\Sigma$ of dimension $n\geq 2$. Then,
    \begin{gather}\label{ineq:isoptype}
    \sigma_k(M_h)|\partial M_h||M_h|^{\frac{2-n}{n}}< (|\Omega||\Sigma|)^{2/n}\lambda_k.
    \end{gather}
    Moreover, for the constant warping function $h=R>0$, the following holds:
    \begin{gather}\label{limitcaseisop}
    \lim_{R\to\infty}\sigma_k(M_R)|\partial M_R||M_R|^{\frac{2-n}{n}}=(|\Omega||\Sigma|)^{2/n}\lambda_k.
    \end{gather}
\end{theorem}

Theorem \ref{thm:isoptype} motivates some natural questions.
For instance, for each fixed index $k$, how large can $\sigma_k(M_h)$ be among all warping functions $h\in \RC^\infty(\Omega)$ such that the boundary $\partial M_h$ has prescribed measure? Or even such that $h$ is constant on $\partial\Omega$? We succeed in solving this problem  for all warped products, with results depending on the dimension $n$ of the fibre $\Sigma$. Indeed, for fibres of low dimension $n=1,2$ the behaviour is peculiar. The case $n=2$ is the only one where the Steklov eigenvalues enjoy a finite upper bound, and we provide the sharp value of this bound in Corollary \ref{coro:initialproject}. In all other dimensions there are no upper bounds independent of the warping function $h$. However, where $h$ remains bounded, the geometries almost realising the supremum for $\sigma_k(M_h)$ have the corresponding eigenfunction oscillating in the base when the fibres have dimension $n = 1$, and in the fibre when the fibres have dimension $n \ge 3$, independently of the dimension $d \ge 2$ of the base, as seen in Theorems \ref{thm:MainIntro} and \ref{thm:onedfibre}.  We also see from their respective proofs that the geometry achieving large eigenvalues is significantly more singular for fibres of dimension $n =1$  than in the case where $n\geq 3$.

Theorem \ref{thm:isoptype} follows from Theorems \ref{thm:upperbound} and \ref{thm:MainIntro}, which provide more general inequalities.

\begin{theorem}\label{thm:upperbound}
Let $M_h=\Omega\times_h\Sigma$ be a warped product of dimension $N \geq 3$ with fibre $\Sigma$ of dimension $n\geq 1$. For each index $k\geq 1$, and each $0<h\in \RC^\infty(\Omega)$, we have that
 \begin{equation} \label{ineq:introgeneral}
\sigma_{k}(M_h)<\lambda_k\frac{\int_{\Omega}h^{n-2}\,dV_\Omega}{\int_{\partial\Omega}h^n\,dA_{\partial\Omega}}.
\end{equation}
Moreover, if the fibre $\Sigma$ has dimension $n \ge 2$, then, for $R \ge 1$, if $h \leq R$ on $\Omega$ and $\int_{\partial \Omega} h^n \, dA_{\partial \Omega} = |\partial \Omega |$, then for each index $k \geq 1$,
\begin{equation}\label{ineq:introupbd}
    \sigma_{k}(M_h)< R^{n-2}\frac{\lambda_k|\Omega|}{|\partial \Omega|}.
\end{equation}
\end{theorem}

The proof of inequality \eqref{ineq:introgeneral} uses separation of variables, which gives a decomposition of the Steklov eigenvalues of $M_h$ as a union of spectra of auxiliary Steklov problems for diffusion operators $L_h f = -\div h^n \nabla f+\lambda_jh^{n-2}f$ on $\Omega$. Once this is available, the proof of inequality \eqref{ineq:introgeneral} is a straightforward use of constant trial functions in the variational characterisation of eigenvalues. The same strategy was used to obtain an upper bound for the first eigenvalue of the aforementioned diffusion operator in \cite[Theorem 2]{Tsu80}. Further properties of this operator were obtained in \cite{Ej79}, where Laplace-isospectral non-isometric closed warped product manifolds were constructed.

In view of Theorems \ref{thm:isoptype} and \ref{thm:upperbound}, we study bounds on the Steklov eigenvalues of warped products given various geometric constraints. We first investigate how large $\sigma_k(M_h)$ can be among all warping functions $h\in \RC^\infty(\Omega)$ such that $M_h$ has a prescribed measure.
    \begin{theorem}\label{thm:boundedinteriorvolume}
        Let $M_h=\Omega\times_h\Sigma$ be a warped product. Then,
$$\sup\{\sigma_k(M_h)\,:\,|M_h|=1\}=+\infty.$$
    \end{theorem}

The next two theorems indicate the precise way in which the eigenvalue functional $\sigma_k(M_h)$ is unbounded under the constraint that the warping function $h=1$ on $\partial \Omega$. The behaviour for one-dimensional fibres differs significantly from that in dimensions three and higher, which is to be expected since the numerator in \eqref{ineq:introgeneral} behaves completely differently in either of those situations. Nevertheless, the starting point for all the proofs is a comparison with cylindrical models, namely products with constant warping function.

\begin{theorem}\label{thm:MainIntro}
Let $M_h=\Omega\times_h\Sigma$ be a warped product of dimension $N \geq 3$ with fibre $\Sigma$ of dimension $n\geq 2$. Let $R\geq 1$.  
For each index $k\geq 1$,
\begin{equation}\label{intro:doubleinequality}
  \sup \{\sigma_k(M_h) : h \le R, \ h =1  \text{ on } \del \Omega\}  = R^n\sigma_k(M_R) < R^{n-2}\frac{\lambda_k\vert \Omega\vert}{\vert \partial \Omega \vert}.
\end{equation}
Moreover,
\begin{equation} \label{Case n>2}
\lim_{R\to \infty} R^2\sigma_k(M_R)=\frac{\lambda_k\vert \Omega\vert}{\vert \partial \Omega \vert},
\end{equation}
so that for $n > 2$, all expressions in \eqref{intro:doubleinequality} go to $\infty$ as $R \to \infty$.
\end{theorem}

For fibres of dimension $1$, the saturated bound reads differently.
\begin{theorem}
\label{thm:onedfibre}
    Let $M_h = \Omega \times_h \Sigma$ be a warped product of dimension $N \geq 3$ with fibre $\Sigma$ of dimension $n=1$. Let $R \ge 1$. Then for each index $k \ge 1$,
    \begin{equation}
    \label{intro:ineqsteklovbase}
\sup\{\sigma_k(M_h)\,:\, h\le R, h=1\ \text{on} \  \partial \Omega\} =R \sigma_k(\Omega).
    \end{equation}
    In particular, this supremum goes to $\infty$ as $R \to \infty$.
\end{theorem}

The bounds \eqref{intro:doubleinequality} and \eqref{intro:ineqsteklovbase}, together with the asymptotics given in \eqref{Case n>2} provide us directly with the next corollary. 
\begin{coro}
\label{coro:initialproject}
Let $M_h=\Omega\times_h\Sigma$ be a warped product of dimension $N \geq 3$ and fibre $\Sigma$ of dimension $n$. Then,
\begin{equation}
\sup\{\sigma_k(M_h)\,:\,h=1 \text{ on }\partial\Omega\}=
\begin{cases}
    \frac{\lambda_k\vert \Omega\vert}{\vert \partial \Omega \vert}&\text{ if } n=2,\\
    +\infty&\text{ if }n \ne 2.
\end{cases}
\end{equation}
\end{coro}

The condition on the total dimension $N$ being at least $3$ stems from general topological bounds on eigenvalues normalised by perimeter when $N = 2$, see \cite[Section 3]{CoGiGoSh2024} for a discussion.
In particular, in \cite[Theorem 1.1]{FaTaYu2015}, Fan, Tam and Yu give a complete answer to this maximisation problem on warped products. For $k = 1$, the maximiser amongst all metrics is in fact known to be the critical catenoid, which happens to be a warped product, by the work of Fraser and Schoen~\cite{FrSc2016}. 
From Theorems~\ref{thm:MainIntro} and~\ref{thm:onedfibre}, it is simple to prove Theorem~\ref{thm:isoptype}.
\begin{proof}[Proof of Theorem \ref{thm:isoptype}]
When $n = 2$, inequality \eqref{ineq:isoptype} can be read directly from \eqref{ineq:introgeneral}, using the observation that $|M_h|=|\Sigma|\int_{\Omega}h^n$ and $|\partial M_h|=|\Sigma|\int_{\partial\Omega}h^n$. The case where $n \ge 3$ follows in a similar way after applying H\"older's inequality with exponent $p=n/(n-2)$.
Finally, the saturation limit \eqref{limitcaseisop} is a consequence of \eqref{Case n>2}.
\end{proof}

\subsection{Stability for two-dimensional fibres}
We record the case $n=2$ of Theorem \ref{thm:MainIntro} separately since it plays a special role. 
\begin{coro}
Let $R \ge 1$ and $h \le R$ be a smooth warping function. Let $M_h=\Omega\times_h\Sigma$ be a warped product of dimension $N \geq 3$ with fibre $\Sigma$ of dimension $n=2$. 
If $h\equiv 1$ on $\partial \Omega$, then for each $k\geq 1$,
\begin{equation} \label{Case n=2}
\sigma_k(M_h) \leq R^2\sigma_k(M_R)<\frac{\lambda_k\vert \Omega\vert}{\vert \partial \Omega \vert},
\end{equation}
and
\begin{equation} \label{Case n=2 sharp}
\sup\{\sigma_k(M_h), h=1\ \text{on} \  \partial \Omega\} =\frac{\lambda_k\vert \Omega\vert}{\vert \partial \Omega \vert} .
\end{equation}
\end{coro}

In our construction, \eqref{Case n=2 sharp} is attained by a sequence of warping functions that become very large everywhere in the interior of $\Omega$. This behaviour is actually required in order to approach the optimal bound.
Indeed, by using more precise trial functions we obtain a quantitative stability improvement of \eqref{Case n=2}.

\begin{theorem} \label{thm:Stability}
Let $M_h=\Omega\times_h\Sigma$ be a warped product of dimension $N \geq 3$ with fibre $\Sigma$ of dimension $n=2$. 
For any warping function $h$ such that
$\int_{\partial \Omega} h^n \, dA_{\partial \Omega} = |\partial \Omega|$, let $\delta(h):= \frac{\lambda_k \vert \Omega\vert}{\vert \partial \Omega\vert}-\sigma_k(M_{h})>0$. If $q\in \Omega$ and $B(q,r) \subset \Omega$ is the ball of radius $r$ and centre $q$, then the following quantitative estimate holds:
$$
\int_{B(q,r)}h^2dV_{\Omega}\ge \frac{\lambda_k^2r^2}{4\delta(h)}\frac{\vert B(q,\frac{r}{2})\vert^2}{\vert\partial \Omega\vert}-\lambda_kr^2 \vert B(q,r)\vert.
$$    
In particular, if $h_\eps$ is a family of warping functions such that $\int_{\partial \Omega} h_\eps^n \, dA_{\partial \Omega} = |\partial \Omega|$ and if
$$\sigma_k(M_{h_\eps})\xrightarrow{\eps\to0}\frac{\lambda_k|\Omega|}{|\partial\Omega|},$$
then for each subdomain $D\subset\Omega$, 
\begin{equation}\label{eq:limstab}
    \lim_{\eps\to 0}\int_{D}h_\eps^2\,dV_{\Omega}=+\infty.
\end{equation}

\end{theorem}

\begin{rem}
    From Theorem~\ref{thm:isoptype}~and Corollary~\ref{coro:initialproject},
    we see that for $n \geq 3$ and each index $k \geq  1$, there exists a family of warping functions $h_\epsilon$ such that $h_\epsilon \equiv 1$ on $\partial \Omega$ and $\sigma_k(M_{h_\epsilon}) \to \infty$ as $\epsilon \to 0$, and so
    $|M_{h_\epsilon}| = |\Sigma| \int_\Omega h_\epsilon^n \to \infty$. In particular,
    $$ \lim_{\eps\to 0}\int_\Omega h_\epsilon^n = +\infty.$$
\end{rem}

\subsection{Constraining the \texorpdfstring{$\RL^p$}{Lp} norm of
\texorpdfstring{$h$}{h}}
If $\int_{\partial \Omega} h^n \, dA_{\partial \Omega} = |\partial \Omega|$, inequality~\eqref{ineq:introgeneral} provides an upper bound in terms of the $\RL^{n-2}$-norm of $h$ while the hypothesis $h\leq R$ is an $\RL^\infty$-norm constraint. H\"older's inequality leads to the following immediate corollary.

\begin{coro}\label{coro: integralbound}
Let $M_h=\Omega\times_h\Sigma$ be a warped product with fibre $\Sigma$ of dimension $n\geq 3$. Suppose that $\int_{\partial \Omega} h^n \dA = |\partial \Omega|$.
If $n-2\leq p<\infty$ then, for each index $k\geq 1$,
\begin{equation} \label{ineg: n ge 3}
\sigma_k(M_h) < \frac{\lambda_k \vert \Omega \vert^{\frac{p-{n+2}}{p}}}{|\partial\Omega|}\|h\|_{\RL^p}^{n-2}.
\end{equation}       
\end{coro}

\begin{rem}
In the case where $h=1$ on $\partial\Omega$, for $p=n-2$, we see that inequality \eqref{ineg: n ge 3} is a special case of inequality \eqref{ineq:introgeneral}. On the other hand, letting $p\to\infty$ in \eqref{ineg: n ge 3} leads to 
$$\sigma_k(M_h) \leq \frac{\lambda_k \vert \Omega \vert}{|\partial\Omega|}\|h\|_{\RL^\infty}^{n-2},$$
which is the non-strict version of inequality \eqref{intro:doubleinequality}.
\end{rem}
To complete the picture provided by Corollary \ref{coro: integralbound} we also investigate the case $1\leq p< n-2$.
\begin{theorem}\label{thm:3dimconnected}
Let $M_h=\Omega\times_h\Sigma$ be a warped product with fibre $\Sigma$ of dimension $n\geq 3$. 
If $1\leq p<n-2$ and if $\partial \Omega$ is connected (in particular, $d\geq 2$), then for each positive constant $R>0$,
\begin{equation}  \label{case large n}
\sup\left\{\sigma_k(M_h) : \int_{\Omega}h^p \dV \leq R; \, h\equiv1\ \text{on}\ \partial \Omega \right\}=+\infty.
\end{equation}
\end{theorem}
The connectedness of the boundary plays an important role in our proof of Theorem \ref{thm:3dimconnected}. 
To illustrate the importance of this condition, we show that when the base $\Omega$ has dimension $d = 1$, so that the boundary cannot be connected, the first eigenvalue remains bounded with any $\RL^p$ constraint on $h$ for $p \ge 1$.
\begin{theorem}\label{dim basis 1}
Let $M_h=[0,L]\times_h\Sigma$ be a warped product with fibre $\Sigma$ of dimension $n\geq 2$. Let $p\ge 1$. 
Suppose that $h(0)=h(L)=1$.
Then
\begin{equation}
    \sigma_1(M_h) \le  \frac{8 n}{L^{\frac{n + p}{p}}}\|h\|_{\RL^p}^n.
\end{equation}
 
\end{theorem}
 
\subsection{Ideas of the proofs and plan of the paper.}
The proofs of the upper bounds proceed in two general steps. First, we use the variational characterisation in order to prove a general bound without constraints. Then, we find a model warping function realising this bound, often a constant warping function, but this warping function does not satisfy the imposed constraints. Finally, we show that we can approximate it with warping functions that do satisfy the constraints, in such a way that the eigenvalues converge and, so, saturate the upper bounds.

In Section \ref{sec:prelim}, we lay the foundations for the proofs of the main theorem by recalling the separation of variables on warped products, investigating the spectrum of the particular case $M_R$ for constant warping functions $R$, and presenting some technical lemmas and conventions. 

Section \ref{sec:control} presents continuity results. First, a relatively simple result allows us to approximate constant warping functions with warping functions having a prescribed boundary behaviour. Then, a homogenisation procedure when the fibres have dimension $n=1$ provides considerable flexibility in terms of the limit problem. This flexibility exhibited by the Steklov spectrum is a general phenomenon, in the sense that it is
possible, by perturbing the boundary in more or less singular ways, to
approximate other eigenvalue problems. For instance, by perforating the interior
of a manifold with many tiny, uniformly distributed holes, the Steklov
eigenvalues converge to those of the Laplacian, potentially with Neumann
boundary conditions \cite{GiHeLa2021,GiKaLa2021,GiLa2021}, or oscillations of the boundary allow one to adjust boundary weights for the Steklov problem \cite{KaLa2022}. As seen in equations \eqref{ineq:introgeneral} and \eqref{RayleighQuotient}, when the fibres have dimension $n=1$, in contrast with the other cases $n\ge 2$, the warping function $h$ appears with competing powers in the numerator of the Rayleigh quotient. While this leads to more singular behaviour that is more technically involved to analyse, it is also precisely what we can exploit using the homogenisation procedure in order to achieve the aforementioned flexibility.

Finally, in Section \ref{proof: theorems}, we obtain the upper bounds and stability results via a combination of variational and continuity methods.

\section*{Acknowledgements}
J.B and B.C acknowledge support of the SNSF project \emph{``Geometric Spectral Theory''}, grant number 200021-19689. A.G. acknowledges the support of the Natural Sciences and Engineering Research Council of Canada (NSERC), in particular through the discovery grant \emph{``Isoperimetry and spectral geometry''} (RGPIN-2022-04247), as well as the support of Fonds de recherche du Québec (FRQ) through the team research project \emph{``G\'eom\'etrie des \'ecarts spectraux pour les probl\`emes de Laplace et de Steklov''}.
K.G. acknowledges support from a Royal Society Research Grant entitled \emph{``Upper bounds for Steklov eigenvalues, ratios, and gaps''}. In addition, this work was partially supported by a grant from the Simons Foundation to K.G.

The authors would like to thank the Isaac Newton Institute for Mathematical Sciences, Cambridge, for their support and hospitality during the program Geometric spectral theory and applications, where work on this paper was undertaken. This work was supported by EPSRC grant EP/Z000580/1.

The authors are grateful to Mikhail Cherdantsev, Charlotte Dietze, Asma Hassannezhad, and Bernard Helffer for useful conversations.

\section{Preliminary material}\label{sec:prelim}

This section is devoted to precisely defining the notation we use in the paper as well as introducing some technical lemmas required in our proofs. We also explicitly obtain some properties of the spectrum for the cylindrical models $M_R = \Omega \times_R \Sigma$ where the warping function is some constant $R > 0$.

\subsection{Asymptotic notation} We make extensive use of the Landau--Vinogradov asymptotic notation in our proofs.
\begin{itemize}
    \item Indiscriminately, we write $a_m = O(b_m)$ or $a_m \lesssim b_m$ to indicate that there exists $C > 0$ so that $|a_m| \le C|b_m|$ for all $m$ sufficiently large. The explicit dependence on the parameter $m$ may sometimes be omitted.
    \item We write $a_m\asymp b_m$ to indicate that $a_m = O(b_m)$ and $b_m = O(a_m)$.
    \item We write $a_m = o(b_m)$ to indicate that $\tfrac{a_m}{b_m} \to 0$ as $m \to \infty$.
    \item We write $a_m = b_m + O(c_m) $ or $a_m = b_m + o(c_m)$ to indicate that $a_m - b_m = O(c_m)$ or $a_m - b_m = o(c_m)$, respectively. 
    \item Sometimes, rather than using the Landau--Vinogradov notation for a numerical sequence, we use it for a sequence of functions depending on other parameters. The notation should be understood as being uniform in the other parameters; where this is crucial we emphasise this fact.
    \item In some situations, the asymptotic notation is used \emph{mutatis mutandis} for a small parameter $t \to 0$ rather than $m \to \infty$.
\end{itemize}

\subsection{Separation of variables on warped products}\label{sec:sepvar}
The warped product structure allows the use of separation of variables. This is expressed in terms of auxiliary Steklov problems for the diffusion operator $L_{h}: \RC^\infty(\Omega) \to \RC^\infty(\Omega)$ on the base, which is defined by
$$
L_{h} f =-\div h^n \nabla f.
$$
For each $\lambda\geq 0$, consider the following auxiliary Steklov problem:
\begin{gather}\label{problem:auxiliary}
\begin{cases}
    L_h f +\lambda h^{n-2} f=0&\mbox{ in }\Omega,\\
    \partial_\nu f=\sigma f&\mbox{ on }\partial \Omega.
\end{cases}
\end{gather}
We make use of the weak formulation of this problem:
\begin{equation}\label{eq:weakproblem}
  \forall\phi\in W^{1,2}(\Omega), \qquad  \int_{\Omega}h^n\nabla f\cdot\nabla\phi+\lambda h^{n-2}f\phi\dV=\sigma\int_{\partial\Omega}h^nf\phi\dA.
\end{equation}
If the dimension of the base $\Omega$ is $d\geq 2$, then for each $\lambda\geq 0$, this problem has discrete unbounded spectrum:
\[0\leq\sigma_{\lambda,0}(h)\leq\sigma_{\lambda,1}(h)\leq\sigma_{\lambda,2}(h)\leq\ldots\nearrow+\infty.\]
Corresponding eigenfunctions $(f_{\lambda,\ell})_{\ell\geq 0  }$ 
normalised to $\int_{\partial\Omega}f_{\lambda,\ell}^2h^n\dA=1$ form a complete orthogonal set.
However, if the base is 1-dimensional, there are only two eigenvalues for each $\lambda$: 
$$\{\sigma_{\lambda,0}(h)\leq\sigma_{\lambda,1}(h)\}.$$
Indeed, the corresponding Dirichlet-to-Neumann map $\Lambda_h$ is a non-negative self-adjoint operator on $\RL^2(\partial\Omega,h^n\dA)$, with compact resolvent.

\begin{proposition}\label{prop:separation}
Let the eigenvalues of the Laplace operator on the fibre $\Sigma$ be $0=\lambda_0<\lambda_1\leq\lambda_2\leq\cdots\nearrow+\infty$.
The Steklov spectrum of $M_h=\Omega\times_h\Sigma$ is given by the multiset
\begin{gather}\label{eq:separatedeigenvalues}
\{\sigma_k(M_h)\}_{k\geq 0}=\{\sigma_{\lambda_j,\ell}(h)\,:\,j,\ell\geq 0\}.
\end{gather}    
\end{proposition}
The proof is standard: it follows by considering products of functions of the form 
$$u_{\lambda_j,\ell}(x,p)=f_{\lambda_j,\ell}(x)\phi_j(p),$$ 
where the functions $f_{\lambda_j,\ell}\in \RC^\infty(\Omega)$ are corresponding eigenfunctions of the above auxiliary Steklov problem~\eqref{problem:auxiliary} for $\lambda=\lambda_j$ and the functions $\phi_j\in \RC^\infty(\Sigma)$ are Laplace eigenfunctions
forming an orthonormal basis of $\RL^2(\Sigma)$. These functions form a complete set of Steklov eigenfunctions on $M_h$. For details, see~\cite{GiPoly2024}, where the operator $L_h$ is multiplied by $h^{-n}$. 
Of particular interest to us is that the eigenvalues $\sigma_{\lambda,\ell}(h)$ admit familiar variational characterisations in terms of the following Rayleigh quotient for $f\in \RW^{1,2}(\Omega)$:
\begin{gather}\label{RayleighQuotient}
\CR_{\lambda,h}(f)=\frac{\int_\Omega |\nabla f|^2h^n+\lambda f^2h^{n-2}\dV}{\int_{\partial\Omega}f^2h^n\dA}\,.    
\end{gather}
We observe that for each $k\geq 0$,
\begin{equation}\label{ineq:kcompk0}
    \sigma_k(M_h) \le \min_{(j+1)(\ell+1) = k+1}\sigma_{\lambda_{j},\ell}(h) \leq \min\{\sigma_{\lambda_k,0}(h),\sigma_{0,k}(h)\}.
\end{equation}
This follows directly from the representation~\eqref{eq:separatedeigenvalues}, since the labelled eigenvalues $\sigma_k(M_h)$ are listed in increasing order by definition. 
In particular, we often use that
\begin{equation}
    \sigma_1(M_h)=\min\{\sigma_{\lambda_1,0}(h),\sigma_{0,1}(h)\}
\end{equation}

We make extensive use of the following characterisation of $\sigma_{\lambda,0}$:
\begin{gather}\label{eq:cara3}
    \sigma_{\lambda,0}(h)=\min\bigg\{\CR_{\lambda,h}(f):f\in \RW^{1,2}(\Omega)\setminus\{0\}\bigg\}=\CR_{\lambda,h}(f_{\lambda,0}),
\end{gather}
with equality if and only if $f$ is a corresponding eigenfunction.

\begin{rem}
    The role of the dimensions of the fibres in Theorem~\ref{thm:MainIntro} is manifest in formula~\eqref{RayleighQuotient}, where different behaviours of $h$ are penalised for dimensions $n=1$ and $n \ge 2$. In particular,  in the proof of Corollary \ref{coro:initialproject} that follows, for fibres of dimension $n \geq 2$ it is the eigenvalues of the form $\sigma_{\lambda_k, 0}(h)$ that play a prominent role, while for fibres of dimension $n=1$ it is the eigenvalues of the form $\sigma_{0,k}(h)$ that do.
     \end{rem}

By using the constant function $f\equiv 1$ as a trial function in~\eqref{eq:cara3}, inequality~\eqref{ineq:kcompk0} leads to
\begin{gather}\label{ineq:consttest}
\sigma_k(M_h)\leq \sigma_{\lambda_k,0}(h)\leq\lambda_k\frac{\int_{\Omega}h^{n-2}\dV}{\int_{\partial\Omega}h^n \dA}.
\end{gather}

\noindent
For arbitrary $\lambda\geq 0$ and index $\ell\geq 0$, we also use the characterisation
\begin{gather}\label{eq:minmaxcharseparated}
\sigma_{\lambda,\ell}(h)= \min_{E\subset \RW^{1,2}(\Omega), \dm (E)=\ell+1} \max\{\CR_{\lambda,h}(u):\ u \in E\}.
\end{gather}

\subsection{Steklov spectrum of \texorpdfstring{$M_R$}{MR}} \label{Spectrum $M_R$} Recall that for $R>0$,
$$M_R=\Omega \times_R \Sigma = (\Omega,g_{\Omega})\times (\Sigma,R^2g_{\Sigma}).$$
In this section we investigate the spectrum of $M_R$ when $R$ is large. This is useful for the proof of Theorem \ref{thm:MainIntro}. In particular, we obtain \eqref{Case n>2} shortly.

For $h=R>0$, problem \eqref{problem:auxiliary} simplifies to
\begin{gather}\label{eq:sepwarpedconstant}
\begin{cases}
    -\Delta f+\frac{\lambda}{R^2}f=0&\mbox{ in }\Omega,\\
    \partial_\nu f=\sigma f&\mbox{ on }\partial \Omega.
\end{cases}
\end{gather}
This is the Steklov problem for the Helmholtz equation with energy $\mu=\lambda/R^2$.
The smallest eigenvalue of this equation is given by
$$\sigma_{\mu,0}=\min\left\{\int_{\Omega}\bigl(|\nabla f|^2+\mu f^2\bigr)\dV\,:\,f:\Omega\to\R, \int_{\partial\Omega}f^2\dA=1\right\}.$$
Because $\sigma_{\mu,0}$ is the smallest eigenvalue, it is simple. Let $f_\mu$ be the corresponding eigenfunction. Then we can compute the derivative with respect to $\mu$,
$$\sigma_{\mu,0}'=2\int_{\Omega}\bigl(\nabla f_\mu\cdot \nabla f_\mu'+\mu f_\mu f_\mu'\bigr)\dV+\int_\Omega f_\mu^2\dV.$$
We claim that the first integral on the right-hand side vanishes. Indeed, the function $f_\mu$ satisfies the weak equation corresponding to \eqref{eq:sepwarpedconstant}, that is \eqref{eq:weakproblem} with $h\equiv R$:
$$\forall \phi\in \RC^\infty(\Omega), \qquad \int_\Omega \bigl(\nabla f_\mu\cdot \nabla \phi+\mu f_\mu \phi\bigr)\,dV_\Omega=\sigma\int_{\partial\Omega}f_\mu \phi\,dA_{\partial\Omega}.$$
Using $\phi=f_\mu'$ as a trial function gives
$$\int_\Omega \bigl(\nabla f_\mu \cdot \nabla f_\mu'+\mu f_\mu f_\mu'\bigr)\dV=\sigma\int_{\partial\Omega}f_\mu f_\mu'\dA.$$
Now, differentiating the normalisation $\int_{\partial\Omega}f_\mu^2=1$ with respect to $\mu$ gives $\int_{\partial\Omega}f_\mu f_\mu' \dA=0$, proving the claim.
It follows that
$$\sigma_{\mu,0}'=\int_\Omega f_\mu^2 \dV=\frac{\int_\Omega f_\mu^2\dV}{\int_{\partial\Omega} f_\mu^2\dA}.$$
In particular, if we compute the derivative at $\mu=0$, using that $f_0$ is constant we get
\begin{gather}\label{derivativeStekMu}
    \frac{d}{d\mu}\sigma_{\mu,0}\Bigl|\Bigr._{\mu=0}=\frac{|\Omega|}{|\partial\Omega|}.
\end{gather}
Now in order to prove \eqref{Case n>2} of Theorem \ref{thm:MainIntro}, we compute:
\begin{gather*}
    \lim_{R\to\infty}R^2\sigma_{\lambda_k/R^2,0}
    =\lim_{R\to\infty}\lambda_k\frac{\sigma_{\lambda_k/R^2,0}-\sigma_{0,0}}{\lambda_k/R^2}
    =\lambda_k\sigma_{0,0}'
    =\lambda_k|\Omega|/|\partial\Omega|.
\end{gather*}
Finally, observe that for  $\lambda=\lambda_k$ with $k\ge 1$, taking the constant function $f=1$ as a trial function leads to
\begin{equation} \label{ineg: C}
    \sigma_{\lambda_k,0}(R) \le \frac{\lambda_k \vert \Omega\vert}{R^2\vert \partial \Omega \vert}.
\end{equation}
\noindent
On the other hand, $\sigma_{\lambda_k,1}(R)\ge \sigma_{0,1}(R)=\sigma_1(\Omega).$
We conclude that, if $R$ is large enough then,
\begin{equation}
    \sigma_j(M_R)=\sigma_{\lambda_j,0}(M_R),\qquad \forall j\in\{0,1,\dots,k\}.
\end{equation}
\noindent
In particular, 
$\lim_{R\to\infty}R^2\sigma_{k}(M_R)=\lim_{R\to\infty}R^2\sigma_{\lambda_k/R^2,0}$, which concludes the proof of 
\eqref{Case n>2} of Theorem \ref{thm:MainIntro}.

\begin{rem}
    One should compare formula \eqref{derivativeStekMu} to the well-known formulas for the derivative of simple Robin eigenvalues with respect to the Robin parameter. See~\cite{BucFreKen2017} for instance.
\end{rem}

\subsection{Almost orthogonality}\label{sec:teclem}
Given a Riemannian manifold $M$, let $N \subset M$ be a closed orientable hypersurface with normal vector $\nu$. For $\delta > 0$, let $T_\delta$ be the tubular neighbourhood of $N$ of width $\delta$, and for $t \in (-\delta,\delta)$ let $N_t$ be the hypersurface parallel to $N$ in the direction $t\nu$. If $\delta>0$ is small enough, $x' \in N$ and $t \in (-\delta,\delta)$ provide the Fermi coordinates for the the point $(x',t)\in T_{\delta}$, which is obtained by following the geodesic starting at $x'$ in direction $t \nu_{x'}$ for time $1$. In those coordinates, the metric satisfies $g=\de t^2 + g_N + \bigo{t}$. Furthermore, following the normal geodesic from $N$ to $N_t$ gives a diffeomorphism $\phi_t : N \to N_t$ and $\phi_t^* g_{N_t} = g_N + \bigo t$.  For a function $f : T_\delta \to \R$, we denote $f_t$ the pullback $\phi_t^* f(\cdot,t)$ as a function on $N$.

 The purpose of the next lemma is to show that any extension to $T_\delta$ of an orthonormal family on $N$ remains almost orthonormal on $N_t$, with defect controlled in terms of the Dirichlet energy.

\begin{lemma}
\label{lemma:deBrunopasproduit}
    Let $N \subset M$ be a closed orientable hypersurface of a compact Riemannian manifold. There exists $\delta>0$ and $C>1$ so that each function $f\in \RC^\infty(T_\delta)$ satisfies the following for all $t\in (-\delta,\delta)$:
    \begin{equation}
    \label{eq:normcontrol}
        \abs{\|f\|_{\RL^2(N)}^2 - \|f\|_{\RL^2(N_t)}^2} \le 
        C t^{1/2} \left( \|f\|_{\RL^2(N)} +  \|\nabla f\|_{\RL^2(T_t)}\right)^2.
    \end{equation}
    In particular, for any $f, \tilde f \in \RC^\infty(T_\delta)$,
    \begin{equation}
    \begin{aligned}
        \abs{\langle f, \tilde f \rangle_{\RL^2(N)} -  \langle f, \tilde f \rangle_{\RL^2(N_t)}} &\le 
        Ct^{1/2}\left(\|f\|_{\RL^2(N)} +  \|\nabla f\|_{\RL^2(T_t)} + \|\tilde f\|_{\RL^2(N)} +  \|\nabla \tilde f\|_{\RL^2(T_t)}\right)^2.
        \end{aligned}
    \end{equation}
\end{lemma}

\begin{rem}
    This lemma also applies where $N$ is a boundary component of a compact manifold, with the difference that the neighbourhood $T_\delta$ is now unidirectional and $N_t$ is only defined for $t > 0$. The proof goes through with the clear modifications.
\end{rem}

\begin{proof}
By the fundamental theorem of calculus along geodesics, for every $x' \in N$ we have that
\begin{equation}
\begin{aligned}
    f(x',t)^2 &= \left(f(x',0) + 
    \int_0^t \frac{\del f}{\del s}(x',s) \de s
    \right)^2  \\
    &= f(x',0)^2 + 
    2f(x',0)
    \int_0^t \frac{\del f}{\del s}(x',s) \de s + 
    \left(
    \int_0^t \frac{\del f}{\del s}(x',s) \de s\right)^2.
    \end{aligned}
\end{equation}
Integrating over $N_t$ on both sides with respect to $\de V_{g_{N_t}}=\de V_{g_N}+O(t)$, the left-hand side is $\|f\|_{\RL^2(N_t)}^2$. On the right-hand side, we have 
\begin{equation}
\int_{N_t} f(x',0)^2 \de V_{g_{N_t}} = (1 + \bigo{t}) \|f\|_{\RL^2(N)}^2.
\end{equation}
Furthermore, using the Cauchy--Schwarz inequality, and the Fermi coordinates to compute the integral over $T_t$, we have that
\begin{equation}
\begin{aligned}
 \int_{N_t} \left(\int_0^t \frac{\del f}{\del s}(x',s) \de s\right)^2 \de V_{g_{N_t}} &\lesssim t \int_N \int_0^t \abs{\frac{\partial f}{\partial s}(x',s)}^2 \de s  \de V_{g_N} \\
 &\lesssim t \int_{T_t} |\nabla f|^2 \de V_M.
 \end{aligned}
\end{equation}
Finally, for the cross terms we also use the Cauchy--Schwarz inequality and the previous computations to obtain
\begin{equation}
\begin{aligned}
    \left|\int_{N_t} f(x',0) \int_0^t \frac{\del f}{\del s}(x',s) \de s \de V_{g_{N_t}}\right| &\lesssim  \|f\|_{\RL^2(N)} \left(\int_N \left(\int_0^t \frac{\del f}{\del s}(x',s) \de s \right)^2 \de V_{g_N}\right)^{1/2} \\
    &\lesssim t^{1/2} \|f\|_{\RL^2(N)} \|\nabla f\|_{\RL^2(T_t)}.
\end{aligned}
\end{equation}
Combining these estimates, along with the trivial $t \lesssim t^{1/2}$, is precisely our claim for the norms, the polarisation identity gives the statement for the inner products.
\end{proof}

\section{Control of the spectrum}
\label{sec:control}
In this section, we approximate the warping function $h$ and prescribe the corresponding eigenvalue behaviour. We take different approaches depending on whether the fibre has dimension $n \geq 2$ or dimension $n=1$.

\subsection{Approximation of warping functions.}

The first statement is that we can approximate a given warping function without changing the boundary behaviour.

\begin{proposition}
\label{prop:continuitybis}
    Let $a < b \in (0,\infty)$ and $r, R \in [a,b]$. Let $h \in \RC^\infty(\Omega;[a,b])$ be such that $h\big|_{\del \Omega} \equiv R$. Let $(r_m)$ be a sequence of positive numbers converging to $r$ and $(h_m : m \in \N)$ be a sequence in $\RC^\infty(\Omega;[a,b])$ such that
    $h_m \big|_{\del \Omega} \equiv r_m$, and such that $h_m \to h$ in $\RL^2(\Omega)$. Then, for all $\lambda \in [0,\infty)$ and $k \in \N$, 
    \begin{equation}
     \lim_{m \to \infty}   \sigma_{\lambda,k}(h_m) = \frac{R^n}{r^n} \sigma_{\lambda,k}(h).
    \end{equation}
\end{proposition}

\begin{proof}
    Fix $\lambda\geq 0$ and denote $\sikm = \sigma_{\lambda,k}(h_m)$ with associated eigenfunction $\ukm$ normalised to $\|\ukm\|_{\RL^2(\del\Omega)} = 1$,  and satisfying the weak equation
\begin{equation}\label{eq:weakproblemukm}
\forall \phi\in \RW^{1,2}(\Omega), \qquad    \int_{\Omega}h_m^n\nabla\ukm\cdot\nabla\phi+\lambda h_m^{n-2}\ukm\phi\dV=\sikm\int_{\partial\Omega}h_m^n\ukm\phi\dA.
\end{equation}
We aim to show convergence of $\ukm$ to some function $\hat{u}_k\in \RW^{1,2}(\Omega)$ which satisfy the weak formulation \eqref{eq:weakproblem} for $\sigma=\frac{R^n}{r^n} \sigma_{\lambda,k}(h)$:
\begin{equation}\label{eq:weaklimitproblem}
    \forall \phi \in \RW^{1,2}(\Omega), \qquad \int_\Omega h^n \nabla \hat{u}_k \cdot \nabla \phi + \lambda h^{n-2} \hat{u}_k \phi \dV = \frac{R^n}{r^n}\sigma_{\lambda,k}(h) \int_{\del \Omega} h^n \hat{u}_k \phi \dA.
\end{equation}
    We proceed in three steps. While it would be conceivable that a more direct proof exists for this proposition, these steps mirror those of the proof 
    of the more technically involved Proposition \ref{prop:homogenisation}, and can be used as a blueprint for that one. 

    \smallskip
    \textbf{Step 1: Existence of limits.}
    Let $E$ be the span of the eigenfunctions associated with $\sigma_{\lambda,0}(h),\dotsc, \sigma_{\lambda,k}(h)$. Using $E$ as a trial space in the characterisation \eqref{eq:minmaxcharseparated} for $\sikm$, we see that
    \begin{equation}
    \label{eq:uppersemicontinuity}
    \begin{aligned}
        \sikm &\le \max_{f \in E \setminus \set 0} 
        \frac{\int_\Omega h_m^n |\nabla f|^2 + \lambda h_m^{n-2}f^2 \dV}{\int_{\del\Omega} f^2 r^n_m \dA} \\
        &\le \frac{R^n}{r^n_m} \sigma_{\lambda,k}(h) + \max_{f \in E \setminus \set 0} 
        \frac{\int_\Omega (h_m^n - h^n) |\nabla f|^2 + \lambda (h_m^{n-2} - h^{n-2})f^2 \dV}{\int_{\del\Omega} f^2 r^n_m \dA}.
        \end{aligned}
    \end{equation}
    On the interval $[a,b]$ the function $t \mapsto t^s$ is Lipschitz for any $s \in \R$, so that $h_m^s \to h^s$ in $\RL^2$. Hence, because $E$ is a finite dimensional subspace of $\RC^\infty(\Omega)$,
\eqref{eq:uppersemicontinuity} implies
\begin{equation}
    \sikm \le \frac{R^n}{r^n} \sigma_{\lambda,k}(h) + \smallo 1.
\end{equation}

In particular, $\sikm$ remains bounded as $m \to \infty$.
Now, using $\phi=u_k^{m}$ as a trial function in \eqref{eq:weakproblemukm}, we see that
    \begin{equation}\label{eq:weakukm}
    \begin{aligned}
        \int_\Omega |\nabla \ukm|^2 \dV &\le \|h_m^{-n}\|_{\RL^\infty} (\int_\Omega h_m^n |\nabla \ukm|^2 + \lambda_k h_m^{n-2}{\ukm}^2\dV)\\
        &= \|h_m^{-n}\|_{\RL^\infty} r_m^n \sikm\\
        &\le (\frac{b}{a})^{n} \sikm.
    \end{aligned}
    \end{equation}
    This implies that $\ukm$ remains bounded in $\RW^{1,2}(\Omega)$, and up to subsequences, there are $\hat{\sigma}_k \in \R$ and $\hat{u}_k \in \RW^{1,2}(\Omega)$ so that $\sikm \to \hat{\sigma}_k$ and $\ukm \wk \hat{u}_k$ weakly in $\RW^{1,2}$. 

\smallskip
\textbf{Step 2: Limits solve the eigenvalue equation.}
    We now show that $\hat{\sigma}_k$ and $\hat{u}_k$ are weak eigenpairs for \eqref{eq:weaklimitproblem}. Since $h_m$ 
    (and $h_m^{-1}$ if $n = 1$) are uniformly bounded in $\RL^\infty$, the linear functionals
\begin{equation}
    \phi \mapsto \int_\Omega h_m^n \nabla \ukm \cdot \nabla \phi + h_m^{n-2} \ukm \phi \dV
\end{equation}
are uniformly bounded in $\RW^{1,2}(\Omega)^*$ with respect to $m$. This means that we can compute the limits in the weak formulation \eqref{eq:weakproblemukm} on the dense subspace of $\RC^\infty(\Omega)\subset\RW^{1,2}(\Omega)$.
We start with convergence of the gradient term. For any $\phi \in \RC^\infty(\Omega)$, 
\begin{equation}\label{eq:weaka}
    \int_\Omega h_m^n \nabla \ukm \cdot \nabla \phi - h^n \nabla \hat{u}_k\cdot \nabla \phi \dV = \int_\Omega (h_m^n - h^n) \nabla \ukm \cdot \nabla \phi
    + h^{n} (\nabla \ukm -\nabla \hat{u}_k) \nabla \phi \dV.
\end{equation}
The second term under the integral goes to zero by the weak convergence  $\nabla \ukm \wk \nabla \hat{u}_k$ in $\RL^2$.
In particular,
\begin{equation}\label{eq:weakb}
   \abs{\int_\Omega (h_m^n - h^n) \nabla \ukm \cdot \nabla \phi \dV} \le \|h_m^n - h^n\|_{\RL^2} \|\ukm\|_{\RW^{1,2}} \|\phi\|_{\RC^1} \xrightarrow{m \to \infty} 0,
\end{equation}
where we have again used that $h_m^s\to h^s$ in $\RL^2$, this time for $s=n$.
We turn our attention to the interior $\RL^2$ norm term. This time, since $\ukm \to \hat{u}_k$ strongly in $\RL^2$,
\begin{equation}\label{eq:weakc}
\begin{aligned}
    \abs{\int_\Omega h_m^{n-2} \ukm \phi - h^{n-2} \hat{u}_k \phi \dV} &= 
    \abs{\int_\Omega (h_m^{n-2} - h^{n-2}) \ukm \phi + h^{n-2} (\ukm - \hat{u}_k) \phi \dV} \\
    &\le \|h_m^{n-2} - h^{n-2}\|_{\RL^2} \|\ukm\|_{\RL^2} \|\phi\|_{\RL^\infty} + \\&\qquad+
     \|h^{n-2}\|_{\RL^2} \|\ukm - \hat{u}_k\|_{\RL^2} \|\phi\|_{\RL^\infty} 
    \\&\xrightarrow{m \to \infty } 0.
    \end{aligned}
\end{equation}
Finally, by weak convergence $\ukm \wk \hat{u}_k$ in $\RW^{1,2}$ and compactness of the boundary trace operator,
\begin{equation}\label{eq:weakd}
    \int_{\del \Omega} (\ukm - \hat{u}_k) \phi \dA \xrightarrow{m \to \infty} 0.
\end{equation}

Inserting \eqref{eq:weakb}, \eqref{eq:weakc}, and \eqref{eq:weakd} in \eqref{eq:weaka}, we see that the pair $(\hat{u}_k,r^nR^{-n} \hat{\sigma}_k)$ satisfies \eqref{eq:weaklimitproblem}, as we wanted

\smallskip
\textbf{Step 3: Limits are non-trivial eigenpairs of the correct index.} We need to prove that $\hat{\sigma}_k = \sigma_{\lambda,k}(h)$, which also implies that $\hat{u}_k$ is an associated eigenfunction. We proceed by induction on $k\in\N$. First, we can read from \eqref{eq:uppersemicontinuity} that $\sigma_0^{(m)} \le \frac{R^n}{r^n}\sigma_{\lambda,0}(h) + o(1)$ so that by the variational characterisation $\sigma_0^{(m)} \to\frac{R^n}{r^n}\sigma_{\lambda,0}(h)$. Suppose that $\sigma_\ell^{(m)} \to \frac{R^n}{r^n}\sigma_{\lambda,\ell}(h)$ for $1 \le \ell \le k-1$, so that also $u_\ell^{(m)} \to u_{\lambda,\ell}(h)$ strongly in $\RL^2(\del \Omega)$. Since the convergence is strong, we have that
\begin{equation}
    \int_{\del \Omega} \hat{u}_k u_{\lambda,\ell}(h) \dA = \lim_{m \to \infty} \int_{\del \Omega} \ukm u_\ell^{(m)} \de A = 0.
\end{equation}
Since $\hat{u}_k$ is orthogonal in $\RL^2(\del \Omega)$ to $u_{\lambda,0}(h),\dotsc,u_{\lambda,k-1}(h)$, we have that $\hat{\sigma}_k \ge \frac{R^n}{r^n}\sigma_{\lambda,k}(h)$. In turn, by \eqref{eq:uppersemicontinuity}, $\hat{\sigma}_k \le \frac{R^n}{r^n}\sigma_{\lambda,k}(h) + \smallo 1$ so that they are equal.
\end{proof}

\subsection{Control via homogenisation for fibres of dimension \texorpdfstring{$1$}{1}}

For fibres of dimension $1$, we show that by considering oscillating metrics,
the Steklov eigenvalues for a warped product can approximate those of a
Dirichlet-to-Neumann map associated with an auxiliary operator corresponding to the Laplacian on the fibres $\Sigma$ multiplied by a positive function $\frac{B}{h}$ on the base. 
More precisely, let $B : \Omega \to [0,\infty)$ be a smooth nonnegative function, and consider the spectral problem %\footnote{Signe du Laplacien.}
    \begin{equation}
    \label{eq:spectralwithB}
        \begin{cases}
            (\Delta + \frac{B}{h} \Delta_\Sigma ) u = 0 & \text{in }M_h, \\
            \del_\nu u = \sigma u & \text{on } \del M_h.
        \end{cases}
    \end{equation}
    The spectrum of \eqref{eq:spectralwithB} is discrete, we write it as $\{\sigma_k(M_h,B)\}_{k\geq 0}$.
    Separation of variables leads to the auxiliary problems
    \begin{equation}
        \label{eq:separationpotential}
        \begin{cases}
            L_h f + \lambda_j ( h^{-1}  + B) f = 0 & \text{in } \Omega, \\
            \del_\nu u = \sigma u & \text{on } \del \Omega.
        \end{cases}
    \end{equation}
    The spectrum of \eqref{eq:separationpotential} is also discrete and is denoted
    $\{\sigma_{\lambda_j,\ell}(h,B): j, \ell \ge 0\}$, where $\{\lambda_j : j \ge 0\}$ is still the spectrum of the Laplacian on the fibre $\Sigma$. Similarly to Proposition \ref{prop:separation}, the spectrum of \eqref{eq:spectralwithB} splits in families:
    \begin{equation}
        \{\sigma_k(M_h,B) : k \ge 0\}=\{\sigma_{\lambda_j,\ell}(h,B)\,:\,j,\ell\geq 0\}.
    \end{equation}
    The eigenvalues $\sigma_{\lambda_j,\ell}(h,B)$ are the successive minima for the Rayleigh quotient
    \begin{equation}\label{eq:rayleighquotienthB}
        \CR_{\lambda_j,h,B}(f) = \frac{\int_\Omega h |\nabla f|^2 + \lambda_j(h^{-1} + B)
            f^2 
        \dV}{\int_{\del \Omega} h f^2 \dA},
    \end{equation}
    they satisfy the corresponding weak formulation:
    \begin{equation}\label{eq:weakproblemhB}
 \forall\phi\in \RW^{1,2}(\Omega), \quad   \int_{\Omega}h\nabla f\cdot\nabla\phi+\lambda (h^{-1}+B)f\phi\dV=\sigma\int_{\partial\Omega}hf\phi\dA.
\end{equation}
    \begin{rem}
        The auxiliary problems \eqref{eq:spectralwithB} and \eqref{eq:separationpotential} still make sense with $B \in \RL^{\frac d 2}(\Omega)$, and can also be approached via homogenisation, see for instance the constructions in \cite{Lo2025}. To simplify the presentation, we choose to keep $B$ smooth as it is sufficient for our purposes.
    \end{rem}
    \begin{proposition}
        \label{prop:homogenisation}
        Let $h, B \in \RC^\infty(\Omega;(0,\infty))$. Suppose that the fibre
        $\Sigma$ has dimension $1$ and that the base $\Omega$ has dimension $d\geq 2$. There is a sequence $\{h_m : m \in \N\}$ in
        $\RC^\infty(\Omega)$ of
        warping functions such that $h_m \le h$ everywhere, for all $p \in
        [1,\infty)$, $h_m \to h$ in $\RL^p$, and, for all $j, k$,
        $\sigma_{\lambda_j,k}(h_m) \to \sigma_{\lambda_j,k}(h,B)$ as $m\to\infty$. Furthermore,
        $|\supp(h_m - h)| \to 0$. 
    \end{proposition}

\begin{rem}
The proof of Proposition \ref{prop:homogenisation} relies on an explicit homogenisation construction
of functions $h_m$ so that $h_m\to h$ in $\RL^p$ but for which $h_m^{-1}$ does not converge to $h^{-1}$. It is this discrepancy
that gives us the flexibility in prescribing the limit, but it is also what makes the analysis more involved. 
After a 'Step 0' consisting in the explicit construction of the warping functions $h_m$, we proceed with the same three steps as in the proof of Proposition \ref{prop:continuitybis}, 
identified clearly so that the reader can refer back to it for the structure if they so wish.
\end{rem}
\textbf{Step 0: Construction of the functions $h_m$.}
For $m \in \N$, let
$\BS^m \subset \Omega$ be an $m^{-1}$-separated set that is at distance at least
$m^{-1}$ from $\del \Omega$, and is maximal with respect to both conditions. The Vorono\u{\i} cell associated to $s \in \BS^m$ is
\begin{equation}
 V_s^m=\{x\in\Omega\,:\,d(x,s)\leq d(x,s')\qquad\forall s'\in S^m\}.
\end{equation}
 Fix $\alpha \in (1,\frac{d}{d-1})$. Let
$\omega_s^m$ be the ball centred at $s$ of radius $m^{-\alpha}$ so that $\omega_s^m \subset V_s^m$ for $m$ large enough,
and let $\BT^m := \bigcup_{s \in \BS^m} \omega_s^m$. The cardinality of $\BS^m$ is $\lesssim m^d$, and so 
\begin{equation}\label{eq:volTm}
    |\BT^m| = \sum_{s \in \BS^m} |\omega_s^m| \lesssim m^{(1 - \alpha)d}
    \to 0.
\end{equation}
See \cite[Section 2]{GiLa2021} for properties of this tessellation.
We now define smooth cut-offs for the sets $\omega_s^m$ that assign them the
same weights as $V_s^m$. Concretely, let $0 \not \equiv \chi$ be a smooth non-negative function,
supported on $[-1,1]$, and for $s \in \BS^m$, define $\chi_s^m : \Omega \to
[0,\infty)$ as
\begin{equation}
    \chi_s^m(x) := \frac{|V_s^m| \chi(m^{\alpha} \dist(x,s))}{\int_{\Omega}
    \chi(m^\alpha \dist(x,s)) \dV}.
\end{equation}
We note that $\chi_s^m$ is supported in $\omega_s^m$ and integrates to
$|V_s^m|$. Because $|V_s^m|\asymp m^{-d}$, it follows from a change of variables that
\begin{gather}\label{ineq:LinfinitynormChi}
  \|\chi_s^m\|_{\RL^\infty} \asymp m^{(\alpha - 1)d}.
\end{gather}
We define the potentials
\begin{equation}
    B_m(x) := \sum_{s \in \BS^m} \chi_s^m(x) B(s).
\end{equation}
They are some form of singular Riemann sum for $B_m$, in the sense that they were constructed specifically so that
for every $s \in \BS_m$,
\begin{equation}
\label{eq:riemsumelement}
\int_{V_s^m} B_m \dV = |V_s^m| B(s),
\end{equation} 
in particular $\int_\Omega B_m \de V \xrightarrow{m \to \infty} \int_\Omega B \dV$ and a fortiori the sequence $B_m$ remains bounded in $\RL^1$.
Everything is now in place to define the approximating warping functions
\begin{equation}
    h_m := \frac{h}{1 + h B_m} \quad \text{so that} \quad
    \frac{1}{h_m} = \frac{1}{h} + B_m.
\end{equation}
It follows from \eqref{ineq:LinfinitynormChi} that $\inf h_m\to 0$, as announced.
We already see that $h_m - h$ is supported on $\BT^m$, so that indeed the volume
of its support vanishes in the limit by \eqref{eq:volTm}. Since $h B_m > 0$, we also see that indeed
$h_m \le h$ everywhere. Let us prove our first approximation result.
\begin{lemma}
    \label{lem:firstconvergence}
    For every $p \in [1,\infty)$, $h_m\to h$ in $\RL^p$. Furthermore, $h_m^{-1}
    \dV \wks (h^{-1}+B) \dV$ as measures.
\end{lemma}
\begin{proof}
 Because $h_m\le h$, the triangle inequality gives the first convergence statement:
    \begin{equation}
        \int_{\Omega} |h_m - h|^p \dV = \int_{\BT^m} |h_m - h|^p \dV \le 2^p
        \|h\|_{\RL^\infty}^p |\BT^m| \to 0. 
    \end{equation}
    For the second convergence statement, let $f \in \RC^0$. We compute
    \begin{equation}
        \label{eq:wksinverses}
        \int_\Omega (h_m^{-1} - h^{-1} - B) f \dV =  
   \int_\Omega (B_m - B) f  \dV. 
    \end{equation}
  Since $f$ and $B$ are uniformly continuous on $\Omega$, on each cell $V_s^m$ we have $B(x) = B(s) + o(1)$ and $f(x) = f(s) + o(1)$, where the $o(1)$ term is independent of $s$. It follows that
  \begin{align*}
    \left|\int_{V_s^m} (B_m-B)f\dV\right|&=
  \left|\int_{V_s^m}(B_m-B(s) + o(1))(f(s)+o(1))\,\dV\right|\\
  &\le
  \left|\int_{V_s^m}(B_m-B(s))f(s) \dV\right| + \int_{V_s^m}o(1) (|B_m| + |B(s)| + |f(s)|) \dV.
  \end{align*}
  The first term is 0 by \eqref{eq:riemsumelement}. For the second, we sum over $s \in \BS^m$ to see
  \begin{align*}
    \sum_{s \in \BS^m}\Bigg|\int_{V_s^m}(B_m -B )f\dV\Bigg|&\leq
 o(1) \Bigg( \int_\Omega |B_m| \dV + \sum_{s \in \BS^m}(|B(s)| + |f(s)|)|V_s^m| \Bigg) 
 \end{align*}
 which converges to $0$, since $B_m$ is a bounded sequence in $\RL^1$.
\end{proof}

\textbf{Step 1: Existence of limits.} 
As in the proof of Proposition \ref{prop:continuitybis}, we now prove the existence of limits to the eigenpairs associated with $h_m$. However, the
more singular situation requires more careful considerations. To make notation more fluid, we now denote by $\sikm$ and $\ukm$ the eigenvalues
and eigenfunctions for Problem \eqref{problem:auxiliary} associated with the warping function $h_m$ and the fibre
eigenvalue $\lambda_j$, i.e. 
\begin{equation}
\sikm :=
\sigma_{\lambda_j,k}(h_m).
\end{equation}
The eigenfunctions are normalised so that
\begin{equation}
\int_{\del \Omega} (\ukm)^2 h\dA = 1. 
\end{equation}

We first provide some a priori estimates. 
\begin{lemma}
    \label{lem:apriori}
    For each fixed index $j,k$, the following bounds hold.
    \begin{enumerate}
        \item The sequence of eigenvalues $\sikm$ is bounded. In fact, it is upper semi-continuous:
        \begin{equation}
            \sikm \le
            \sigma_{\lambda_j,k}(h,B) + o(1).
        \end{equation}
        \item The sequence of norms of weighted gradients $\|h_m^{1/2} \nabla
            \ukm\|_{\RL^2(\Omega)}$ is bounded.
            \item The sequence of norms of gradients on the complement of $\BT^m$,
            $\|\nabla \ukm \|_{\RL^2(\Omega \setminus \BT^m)}$, is bounded. 
            \item The sequence of norms of gradients on $\BT^m$ may not be bounded, but it satisfies the bound
            \begin{equation}
                \|\nabla \ukm \|_{\RL^2(\BT^m)} \lesssim m^{\frac{(\alpha - 1)d}{2}}.
            \end{equation}
    \end{enumerate}
\end{lemma}

\begin{proof}
Points (3) and (4) follow directly from (2) along with the observation that $h_m$ is bounded from below
on $\Omega \setminus \BT^m$ and can be as small as $m^{(1 - \alpha)d}$ on $\BT^m$. 
   
    We show boundedness of the sequence of eigenvalues and proceed as in \eqref{eq:uppersemicontinuity}. Let $E$ be the span of the first
   $k+1$ eigenfunctions for Problem \eqref{eq:separationpotential} with
    $\lambda_j$.  Let $SE=\{f\in E\,:\,\int_{\partial\Omega} hf^2  \dA=1\}$. 
    Using that $h_m\leq h$, it follows from the variational characterisation of $\sikm$ that
    \begin{equation}
        \begin{aligned}
            \sikm &\le \sup_{f \in SE}  \int_\Omega h_m |\nabla f|^2 + \lambda_j h_m^{-1} f^2 \dV \\
            &\le \sigma_{\lambda_j,k}(h,B) + \sup_{f \in SE}  \lambda_j\left|\int_\Omega(h_m^{-1} - h^{-1} - B)f^2 \dV\right|.
    \end{aligned}
    \end{equation}
    Because $E$ is a fixed finite dimensional subspace, Lemma
    \ref{lem:firstconvergence} tells us that the last supremum converges to $0$,
    providing the bound on the eigenvalues.

    We now insert $\ukm$ in its own variational characterisation to see that
    \begin{equation}
        \int_\Omega h_m |\nabla \ukm|^2 +  \lambda_j h_m^{-1} (\ukm)^2 \dV = \sikm.
    \end{equation}
    This directly implies that $\|h_m^{1/2} \nabla \ukm\|_{\RL^2(\Omega)}$ is a bounded sequence for any $\lambda_j \in [0,\infty)$. 
\end{proof}
The previous a priori estimates already give us a limit (up to subsequences) $\hat \sigma_k$ for $\sikm$. The sequence $\ukm$ remains too singular to study directly, to this end we define another
sequence $\utkm$ that has better convergence properties and whose distance to $\ukm$ in $\RL^2$ goes to $0$ as $m \to \infty$.
\begin{definition}
    For each $m \in \N$, let $H_m : \RW^{1,2}(\Omega \setminus \BT^m) \to \RW^{1,2}(\Omega)$ be the harmonic extension operator. We define the approximating sequence
    \begin{equation}
        \utkm := H_m(\ukm \boldsymbol{1}_{\Omega \setminus \BT^m})
    \end{equation}
    and the corrector $\phkm$, supported on $\BT^m$, given by
    \begin{equation}\label{eq:defcorrectors}
        \phkm := \ukm - \utkm \in \RW^{1,2}_0(\BT^m).
    \end{equation}
\end{definition}

\begin{lemma}
\label{lem:unifboundutkm}
    The approximating sequence $\utkm$ is bounded in $\RW^{1,2}(\Omega)$, and the corrector sequence $\phkm$ enjoys the bound
    \begin{equation}
    \label{ineq:gradphik}
        \|\nabla \phkm\|_{\RL^2(\BT^m)} \lesssim m^{\frac{(\alpha - 1)d}{2}}.
    \end{equation}
\end{lemma}
\begin{proof}
The harmonic extension operators are uniformly bounded, see
\cite[Lemma 3.4]{GiLa2021}. More precisely, the $\RL^2$ norm of the
gradient of the extension $\utkm$ is bounded in terms of the $\RL^2$ norm
of the gradient of $\ukm$ on $\Omega\setminus\BT^m$
\cite[p. 40, Example 1]{RaTa1975}. Hence, by Lemma
\ref{lem:apriori}(3),
$$\|\nabla \utkm\|_{\RL^2(\Omega)}
    \lesssim
    \|\nabla \ukm\|_{\RL^2(\Omega\setminus\BT^m)}
    \lesssim 1.$$
Moreover, $\utkm$ and $\ukm$ have the same trace on $\del\Omega$, so that
$$\int_{\del\Omega} h(\utkm)^2\,\dA = 1.$$
Since $h$ is positive and bounded away from zero on $\del\Omega$, it
follows that $\|\utkm\|_{\RL^2(\del\Omega)}$ is uniformly bounded.
By equivalence of norms on $\RW^{1,2}(\Omega)$, we therefore obtain
\begin{equation}\label{eq:uniformW12utkm}
    \|\utkm\|_{\RW^{1,2}(\Omega)} \lesssim 1.
\end{equation}
    By the triangle inequality, we can therefore deduce using (4) in Lemma \ref{lem:apriori}, that
    \begin{equation}
        \|\nabla \phkm\|_{\RL^2(\BT^m)} \le \|\nabla \ukm\|_{\RL^2(\BT^m)} + \|\nabla \utkm\|_{\RL^2(\BT^m)} \lesssim m^{\frac{(\alpha - 1)d}{2}}.
    \end{equation}
\end{proof}
We require more precise estimates on the $\RL^2$ norm of the correctors $\phkm$, using Poincaré's inequality.
\begin{lemma}\label{lemma:phkmTm}
    The sequence $\|h_m^{1/2}\nabla \phkm\|_{L^2(\Omega)}$ is bounded. In addition, as $m\to\infty$,
        \begin{equation}\label{limit:hmphkm}
        \|h_m^{-1/2} \phkm\|_{\RL^2(\Omega)} \to 0
    \end{equation}
    and
    \begin{equation}\label{limit:phkm}
      \|\phkm\|_{\RL^2(\BT^m)} \to 0.
    \end{equation}
\end{lemma} 
\begin{proof} 
Because $h_m \le h$,
\begin{equation}
\begin{aligned}
\|h_m^{1/2}\nabla\phkm\|_{\RL^2(\Omega)} 
&\leq\|h_m^{1/2}\nabla \ukm\|_{\RL^2(\Omega)} +\|h^{1/2}\|_{\RL^\infty}\|\nabla \utkm\|_{\RL^2(\Omega)},
\end{aligned}
\end{equation}
and it follows from (2) in Lemma \ref{lem:apriori} and Lemma \ref{lem:unifboundutkm} that this sequence is bounded. 

We now move to the proof of \eqref{limit:hmphkm}.
    For each $s \in \BS^m$, and $f \in \RW^{1,2}_0(\omega_s^m)$, we have the
    Poincar\'e inequality
    \begin{equation}
        \|f\|_{\RL^2(\omega_s^m)}^2 \lesssim m^{-2\alpha} \|\nabla
        f\|_{\RL^2(\omega_s^m)}^2,
    \end{equation}
    where the power $-2\alpha$ is obtained by scaling down the Poincar\'e
    inequality on a ball of radius $1$. Summing both sides over $s \in \BS^m$
    we obtain that for $f \in \RW^{1,2}_0(\BT^m)$,
    \begin{equation}
        \|f\|_{\RL^2(\BT^m)} \lesssim m^{-\alpha} \|\nabla f\|_{\RL^2(T^m)}. 
    \end{equation}
    Using \eqref{ineq:LinfinitynormChi}, applying the Poincar\'e inequality to
    $\phkm$, and using \eqref{ineq:gradphik} gives
    \begin{equation}\label{eq:phikrate}
        \begin{aligned}
        \|h_m^{-1/2} \phkm\|_{\RL^2(\Omega)} &\lesssim m^{\frac{(\alpha -
        1)d}{2}}\|\phkm\|_{\RL^2(\BT^m)} \\&\lesssim m^{\frac{(\alpha - 1)d}{2} -
        \alpha} \|\nabla \phkm \|_{\RL^2(\BT^m)} \\&\lesssim m^{\alpha(d-1) -
        d}.
    \end{aligned}
    \end{equation}
    Since we chose $\alpha < \frac{d}{d-1}$, this indeed goes to $0$. Finally, we also have
    \begin{equation} 
        \|\phkm\|_{\RL^2(\BT^m)}\leq \|h_m^{1/2}\|_{L^\infty}\|h_m^{-1/2} \phkm\|_{\RL^2(\BT^m)}\leq \|h^{1/2}\|_{L^\infty}\|h_m^{-1/2} \phkm\|_{\RL^2(\BT^m)} \to 0.
    \end{equation}
\end{proof}

\begin{lemma}
\label{lem:strongconv}
    There is $\hat u_k$ in $\RW^{1,2}(\Omega)$ so that, up to subsequences, $\utkm$ converges weakly in $\RW^{1,2}$ to $\hat u_k$, and so that $\ukm$ and $\utkm$ converge strongly in $\RL^2$ to $\hat u_k$.
\end{lemma}

\begin{proof}
    We have from Lemma \ref{lem:unifboundutkm} that $\utkm$ is a bounded sequence in $\RW^{1,2}(\Omega)$, so that there is, after extracting a subsequence, a weak limit $\hat u_k \in \RW^{1,2}(\Omega)$ for $\utkm$ as $m \to \infty$. By the Rellich--Kondrachov theorem, $\utkm$ converges strongly in $\RL^2$ to $\hat u_k$. Since $\ukm  = \utkm + \phkm$ and we established in Lemma \ref{lemma:phkmTm} that $\phkm \to 0$ strongly in $\RL^2$, we see that $\ukm \to \hat u_k$ strongly in $\RL^2$. 
\end{proof}

\smallskip
\textbf{Step 2: Limits solve the weak eigenvalue equation.}
Now that we have limits $\hat u_k$ and $\hat \sigma_k$, we prove that they are the solution to
a differential equation.
\begin{proposition}
    The limit function $\hat u_k$ is not identically $0$ and solves the weak
    differential equation
    \begin{equation}
        \forall f \in \RW^{1,2}(\Omega), \qquad \int_\Omega h \nabla \hat u_k \cdot
        \nabla f + \lambda_j(h^{-1} + B) \hat u_k f \dV = \hat \sigma_k \int_{\del \Omega} \hat u_k f h
        \dA.
    \end{equation}
\end{proposition}

\begin{proof}
Since $\ukm$ and $\utkm$ share their boundary values, we see that
    $\|h^{1/2} \utkm\|_{\RL^2(\del \Omega)} = 1$ for all $m \in \N$. Since $\utkm \wk \hat u_k$
    in $\RW^{1,2}(\Omega)$, the boundary traces converge strongly in $\RL^2(\partial\Omega)$, by
    compactness of the boundary trace operator.
    This shows that for all $f \in \RW^{1,2}(\Omega)$,
    \begin{equation}
        \lim_{m \to \infty} \sikm \int_{\del \Omega} \ukm f  h \dA = \hat \sigma_k
        \int_{\del \Omega} \hat u_k f h \dA.
    \end{equation}
    In addition, $\|h^{1/2} \hat u_k\|_{\RL^2(\partial\Omega)}=1$, and in particular $\hat u_k$ does not vanish
    identically. 
    Furthermore, since the sequence $\sikm$ is bounded in $m$,
     the sequence of linear functionals
    \begin{equation}
        f \mapsto \int_\Omega h_m \nabla \ukm \cdot \nabla f +
        \lambda_j h_m^{-1} \ukm f \dV = \sikm \int_{\del \Omega} \ukm f h \dA
    \end{equation}
    is bounded in $\RW^{1,2}(\Omega)^*$. Whence, convergence can be tested
    on a dense subspace of $\RW^{1,2}(\Omega)$, and so we now assume that $f \in
    \RC^\infty(\Omega)$. Next, we consider the sequence
    \begin{equation}
        \label{eq:gradconverges}
        \begin{aligned}
        \int_\Omega h_m \nabla \ukm \cdot \nabla f - h \nabla \hat u_k \cdot \nabla f \dV 
        &= \int_\Omega h_m \nabla \phkm \cdot \nabla f \dV  \\ 
        & \qquad + \int_\Omega (h_m - h) \nabla \utkm \cdot \nabla f \dV  \\
        &\qquad + \int_\Omega h \nabla (\utkm - \hat u_k) \cdot \nabla f \dV,
    \end{aligned}
    \end{equation}
    which we aim to show converges to zero. We proceed term by term.
     For the first term, using that $f \in \RC^\infty(\Omega)$, the Cauchy--Schwarz inequality leads to
    \begin{equation}
        \left|\int_\Omega h_m \nabla \phkm \cdot \nabla f \dV\right| \le
        \|h_m^{1/2}\|_{\RL^\infty(\BT^m)} \|h_m^{1/2}\nabla
        \phkm\|_{\RL^2(\BT^m)} \|\nabla f\|_{\RL^\infty(\Omega)} \Vol(\BT^m)^{1/2}.
    \end{equation}
    Since $0\leq h_m\leq h$, it follows from Lemma \ref{lemma:phkmTm} and \eqref{eq:volTm} that the right-hand side tends to 0. 
    For the second term on the right-hand side of \eqref{eq:gradconverges}, by the
    Cauchy--Schwarz inequality and Lemmas \ref{lem:firstconvergence}, \ref{lem:apriori} and \ref{lem:unifboundutkm}
    \begin{equation}
        \begin{aligned}
         \left| \int_\Omega (h_m - h) \nabla \utkm \cdot \nabla f \dV \right| \le
         \|h_m - h\|_{\RL^2(\Omega)} \|\nabla \utkm\|_{\RL^2(\Omega)} \|\nabla
         f\|_{\RL^\infty(\Omega)}\to 0.
        \end{aligned}
    \end{equation}
    For the third term, weak
    convergence $\utkm \wk \hat u_k$ in $\RW^{1,2}(\Omega)$ implies weak $\RL^2$ convergence of the gradients, so that
    \begin{equation}
        \int_\Omega h \nabla (\utkm - \hat u_k) \cdot \nabla f \dV \to 0.
    \end{equation}
    Altogether this yields that the left-hand side of \eqref{eq:gradconverges} goes to $0$.
    Finally, we consider the sequence
    \begin{equation}
    \label{eq:potentialconverges}
        \begin{aligned}
        \int_\Omega h_m^{-1} \ukm f - (h^{-1} + B) \hat u_k f \dV &=
        \int_\Omega (h_m^{-1} - h^{-1} - B) \ukm f \dV \\
        &\qquad + \int_\Omega (h^{-1} + B) (\ukm - \hat u_k) f \dV,
    \end{aligned}
    \end{equation}
    which we once again show converges to zero. 
By Lemma \ref{lem:strongconv},
    \begin{equation}
      \left|\int_\Omega (h^{-1} + B) (\ukm - \hat u_k) f \dV\right| \le
      \|h^{-1} + B\|_{\RL^\infty} \|\ukm - \hat u_k\|_{\RL^2} \|f\|_{\RL^2} \to 0. 
    \end{equation}
    We are left to analyse the first term on the right-hand side of \eqref{eq:potentialconverges}. Since 
    $B$ is smooth, we have uniformly in each Vorono\u{\i} cell $V_s^m$ that 
    $B(x) = B(s) + O(m^{-1})$. This allows us to rewrite,
    \begin{equation}\label{eq:decomphmhb}
        \begin{aligned}
        \int_\Omega (h_m^{-1} - h^{-1} - B) &\ukm f \dV = \int_\Omega (B_m -
        B) \ukm f  \dV\\
        &\lesssim \sum_{s \in \BS^m} \left(B(s)  \left|\int_{V_s^m}  (\chi_s^m-1) \ukm f \dV\right| + m^{-1} \int_{V_s^m} |\ukm f| \dV \right).
    \end{aligned}
    \end{equation}
    In turn,
    the last term of \eqref{eq:decomphmhb} can be bounded as
    \begin{equation}
        m^{-1}  \sum_{s \in \BS^m} \int_{V_s^m} |\ukm f| \dV \le m^{-1} \|f\|_{\RL^\infty(\Omega)} \|\ukm\|_{\RL^1(\Omega)}
        \to 0.
    \end{equation}
Finally, we decompose $\ukm = \utkm + \phkm$ and use \eqref{ineq:LinfinitynormChi} to obtain
\begin{equation}
\begin{aligned}
    \sum_{s \in \BS^m} B(s)  \bigg| \int_{V_s^m}  (\chi_s^m-1) \ukm f \dV\bigg| &\lesssim \sum_{s \in \BS^m} B(s) \left( \bigg| \int_{V_s^m}  (\chi_s^m-1) \utkm f \dV\bigg| \right.\\&\qquad\quad+ \left.m^{(\alpha - 1)d} \|f\|_{\RL^\infty(\Omega)} \int_{\omega_s^m} |\phkm| \dV\right).
    \end{aligned}
\end{equation}
These are the final two terms whose limits need to vanish as $m \to \infty$. We start by using the Cauchy--Schwarz inequality on integrals then on sums to see that
\begin{equation}
\begin{aligned}
    \sum_{s \in \BS^m}   B(s) \|f\|_{\RL^\infty} \int_{\omega_s^m} m^{(\alpha - 1) d} |\phkm| \dV &\lesssim \sum_{s \in \BS^m} m^{\frac{(\alpha - 1) d}{2}}|\omega_s^m|^{1/2} \|m^{\frac{(\alpha - 1)d}{2}}\phkm\|_{\RL^2(\omega_s^m)} \\
    &\le \|m^{\frac{(\alpha - 1)d}{2}} \phkm\|_{\RL^2(\BT^m)} \xrightarrow{m \to \infty} 0, 
    \end{aligned}
\end{equation}
where we used \eqref{eq:phikrate} for the final limit. On the other hand, notice that since $1 - \chi_s^m$ has zero average on $V_s^m$, which is the solvability condition for the Neumann problem,
    the Lax--Milgram lemma ensures, for each $s \in \BS^m$, the existence of
    $\psi_s^m \in \RW^{1,2}(V_s^m)$ satisfying the weak differential equation
    \begin{equation}
    \label{eq:psivariational}
        \forall v \in \RW^{1,2}(V_s^m), \qquad \int_{V_s^m} \nabla \psi_s^m
        \cdot \nabla v \dV = \int_{V_s^m}  (1 - \chi_s^m)  v \dV.
    \end{equation}
    We can furthermore assume that $\int_{V_s^m} \psi_s^m \dV = 0$. This leads
    us to
      \begin{equation}
        \begin{aligned}
        \left|\sum_{s \in \BS^m} B(s) \int_{V_s^m} (1 - \chi_s^m) \utkm f
        \dV\right| &=
        \left|\sum_{s \in \BS^m} B(s) \int_{V_s^m} \nabla \psi_s^m \cdot \nabla
        (\utkm f) \dV\right| \\
        &\le \|B\|_{\RL^\infty} \|f\|_{\RC^1} \sum_{s \in \BS^m}
        \|\nabla \psi_s^m\|_{\RL^2(V_s^m)} \|\utkm\|_{\RW^{1,2}(V_s^m)} \\
        &\lesssim \|B\|_{\RL^\infty} \|f\|_{\RC^1} 
            \left( \sum_{s \in \BS^m}
            \|\nabla \psi_s^m\|_{\RL^2(V_s^m)}^2 \right)^{1/2}, \\
        \end{aligned}
    \end{equation}
    where we applied the Cauchy--Schwarz inequality on integrals, then on sums, as well as \eqref{eq:uniformW12utkm}.   Our goal is therefore to prove that $\|\nabla \psi_s^m\| = o(m^{-\frac{d}{2}})$. In
    order to do so, we put $\psi_s^m$ in its own variational formulation \eqref{eq:psivariational} and use
    the fact that it has average zero on $V_s^m$, that $\chi_s^m$ is
    supported on $\omega_s^m$, and \eqref{ineq:LinfinitynormChi} to see that for any $p \in (1,\infty)$, 
    \begin{equation}
        \label{eq:amp}
        \begin{aligned}
        \|\nabla \psi_s^m\|_{\RL^2(V_s^m)}^2 &=  -\int_{V_s^m} \psi_s^m \chi_s^m
        \dV  \\
        &\le \|\chi_s^m\|_{\RL^\infty} |\omega_s^m|^{\frac{p-1}{p}}
        \|\psi_s^m\|_{\RL^p(\omega_s^m)} \\
        &\lesssim m^{\frac{(\alpha - p)d}{p}} \|\psi_s^m\|_{\RL^{p}(V_s^m)}.
    \end{aligned}
    \end{equation}
    Suppose now that we choose $p \le \frac{2d}{d-2}$ ($p < \infty$ for $d = 2$)
    and scale the Sobolev embedding $\RW^{1,2} \hookrightarrow \RL^{p}$ to see
    \begin{equation}
        \label{eq:sob}
        \|\psi_s^m\|_{\RL^p(V_s^m)} \lesssim m^{\frac{(p-2)d}{2p}}
        \|\psi_s^m\|_{\RL^2(V_s^m)} + m^{\frac{(p-2) d -2p}{2p}} \|\nabla
        \psi_s^m\|_{\RL^2(V_s^m)}.
    \end{equation}
    Since $V_{s^m}$ is a geodesically convex set of inradius approximately $m^{-1}$, its first Neumann eigenvalue is of the order $m^{2}$. Thus, since $\psi_s^m$ has average $0$ on $V_s^m$, we have that
    \begin{equation}
        \label{eq:pw}
        \|\psi_s^m\|_{\RL^2(V_s^m)} \lesssim m^{-1} \|\nabla
        \psi_s^m\|_{\RL^2(V_s^m)}.
    \end{equation}
    Inserting \eqref{eq:sob} and \eqref{eq:pw} into \eqref{eq:amp}, we see that
    \begin{equation}
        \|\nabla \psi_s^{m}\|_{\RL^2(V_s^m)} \lesssim m^{-\frac d 2 +
        \frac{\alpha d - d - p}{p}} = o(m^{-\frac d 2})
    \end{equation}
    as long as we choose $p > \frac{d}{d-1}$, since we chose $\alpha <
    \frac{d}{d-1}$. This was the last required estimate, completing the proof.
\end{proof}

\textbf{Step 3: Limits are eigenpairs of the correct index.}
    We now know that the limits $\hat u_k$ and $\hat \sigma_k$ form a non-trivial eigenpair
    for Problem \eqref{eq:separationpotential}. It only remains to prove that
    $\hat \sigma_k = \sigma_{\lambda_j,k}(h,B)$ and that whenever
    \begin{equation}
        \sigma_{\lambda_j,k-1} < \sigma_{\lambda_j,k} = \dotso = \sigma_{\lambda_j,k+\ell-1} < \sigma_{\lambda_j,k+\ell},
    \end{equation}
    then $\{\hat u_{k},\dotsc, \hat u_{k+\ell - 1}\}$ forms an orthonormal basis in $\RL^2(\del \Omega, h \dA)$ of the $\ell$-dimensional eigenspace associated to $\sigma_{\lambda_j,k}$.

    By strong convergence of $h_m^{1/2} \ukm \to h^{1/2} \hat u_k$ in $\RL^2(\del \Omega)$, we have that $(\hat u_k : k \in \N_0)$ forms an orthonormal set in
    $\RL^2(\del \Omega, h \dA)$. In particular, for any $k \in \N$, put $E_k = \spn(\hat u_0, \dotsc, \hat u_k)$, which is a $k+1$ dimensional subspace of $\RL^2(\del \Omega, h\dA)$ and from the variational characterisation of $\sigma_{\lambda_j,k}$ we have that
    \begin{equation}
        \sigma_{\lambda_j,k} \le \max_{f \in E_k \setminus \{0\}} R_{\lambda_j,h,B}(f) \le \max_{0 \le \ell \le k} \hat \sigma_\ell.
    \end{equation}
Combining this with the upper semicontinuity statement (1) of Lemma \ref{lem:apriori}, which tells us that $\hat \sigma_k \le \sigma_{\lambda_j,k} + \smallo 1$, we deduce that $\sigma_{\lambda_j,k} = \hat \sigma_k$ for all $k$. In turn, $\hat u_k$ is an eigenfunction for $\sigma_{\lambda_j,k}$, and since they form an orthonormal set they span each eigenspace.    
    This completes the proof of Proposition~\ref{prop:homogenisation}.  \qed

\section{Eigenvalue bounds} \label{proof: theorems}

In this section, we investigate eigenvalue bounds. In the first instance, bounds are obtained via the variational characterisation. Then, we saturate those bounds using the continuity results of the previous section. Finally, we prove the stability result for fibres of dimension $2$ and the lack of universal bounds under $\RL^p$ constraints. 

\subsection{Bounds via the variational characterisation.}

\begin{proof}[Proof of Theorem \ref{thm:upperbound}]
    We start with \eqref{ineq:introgeneral}. In~\eqref{ineq:consttest} we have already obtained that 
\begin{equation}
\label{eq:nonstrict}
\sigma_k(M_h)\leq \lambda_k\frac{\int_{\Omega}h^{n-2}\,dV_\Omega}{\int_{\partial\Omega}h^n\,dA_{\partial\Omega}},
\end{equation}
and we now claim that the inequality has to be strict. Indeed, if $\sigma_{\lambda_k,0}(h)=\lambda_k\frac{\int_{\Omega}h^{n-2}\,dV_\Omega}{\int_{\partial\Omega}h^n\,dA_{\partial\Omega}}$, then $f\equiv 1$ must be an eigenfunction corresponding to $\sigma_{\lambda_k,0}(h)$. However, the strong form of the eigenvalue equation gives:
\begin{equation}
0=L_hf+\lambda_k h^{n-2}f=\lambda_k h^{n-2} f,
\end{equation}
whence $\lambda_k=0$, contradicting the connectedness hypothesis on $\Sigma$. 
From here, we see that \eqref{ineq:introupbd} follows from inequality \eqref{ineq:introgeneral}, using that $h \leq R$ and $\int_{\partial \Omega} h^n \dA = |\partial \Omega |$.
\end{proof}
Immediately, we get the following.

\begin{proof}[Proof of Corollary \ref{coro: integralbound}.]
Since $\int_{\partial \Omega} h^n \, dA_{\partial \Omega} = |\partial \Omega|$, $p \ge n-2 \ge 1$ and $|\Omega| < \infty$, using H\"older's inequality
with  conjugate exponents $\frac{p}{n-2}$ and $\frac{p}{p - n + 2}$ on \eqref{ineq:introgeneral} gives
\begin{equation*}
  \sigma_k(M_h) <  \lambda_k \frac{\int_{\Omega} h^{n-2} dV_{\Omega}}{|\del \Omega|} \le \lambda_k\frac{ \Vert h\Vert_{\RL^p}^{n-2} \vert \Omega \vert^{(p-n+2})/p}{\vert \partial \Omega \vert}.
\end{equation*}
\end{proof}

Using the variational characterisation, we also prove the implicit upper bound in \eqref{intro:doubleinequality}. More precisely, we prove the following.

\begin{lemma}
\label{lem:doubleinequality}
    Let $R \ge 1$. Let $M_h = \Omega \times_h \Sigma$ be a warped product of dimension $N \geq 3$ with fibre $\Sigma$ of dimension $n \ge 2$ where $h \le R$ and $h \equiv 1$ on $\del \Omega$. Then, for each index $k \ge 1$,
    \begin{equation}
    \label{ineq:constantwarping}
        \sigma_k(M_h) \le R^n \sigma_k(M_R) < R^{n-2} \frac{\lambda_k|\Omega|}{|\del \Omega|}.
    \end{equation}
\end{lemma}

\begin{proof}
For any $f\in \RW^{1,2}(\Omega)$ and fixed $\lambda > 0$, since $h=1$ on $\partial\Omega$, we have
\begin{equation}
\label{eq:rayleighmonotone}
\begin{aligned} 
\CR_{\lambda,h}(f)&=\frac{\int_\Omega (|\nabla f|^2 h^n+\lambda f^2 h^{n-2})\dV}{\int_{\partial\Omega} f^2\dA} \\
& \le \frac{\int_\Omega (|\nabla f|^2 R^n+\lambda f^2 R^{n-2})\dV}{\int_{\partial\Omega}f^2\dA}
\\ &=R^n\frac{\int_\Omega (|\nabla f|^2R^n+R^{n-2}\lambda f^2 )\dV}{\int_{\partial\Omega}R^nf^2\dA} \\
&= R^n \CR_{\lambda,R}(f).
\end{aligned}
\end{equation}
Therefore, for every $(j,\ell) \in \N^2_0$, monotonicity of the Rayleigh quotients implies that $\sigma_{\lambda_j,\ell}(h) \le R^n \sigma_{\lambda_j,\ell}(R)$, which provides the first inequality in \eqref{ineq:constantwarping} after rearranging the eigenvalues in increasing order.

  From there, we use \eqref{ineq:kcompk0} to see that $R^n \sigma_k(M_R) \le R^n \sigma_{\lambda_k,0}(R)$, and using $f \equiv 1$ as a trial function for $\sigma_{\lambda_k,0}(R)$ gives
   $$\sigma_k(M_R) \leq R^{n-2}\frac{\lambda_k \vert \Omega \vert}{\vert \partial \Omega \vert}.$$
  The inequality is then strict because constant functions cannot be eigenfunctions of $\sigma_{\lambda,k}$ for $\lambda > 0$,  by the same argument as in the proof of Theorem \ref{thm:upperbound}. 
  \end{proof}

Finally, in order to prove Theorem \ref{thm:onedfibre}, we also need to have a comparison lemma between the Steklov eigenvalues of $M_h$ and the eigenvalues of the limit homogenised problem \eqref{eq:separationpotential}.

\begin{lemma}
\label{lem:lowerBbound}
    Let $\Omega \times_h \Sigma$ be a warped product of dimension $N \geq 3$ with fibre $\Sigma$ of dimension $1$. For all $\lambda > 0$, $\sigma_{\lambda,0}(h,B) \to \infty$ as $B \to \infty$. In particular, for all $k \in \N$, there is $B \in (0,\infty)$ so that $\sigma_k(M_h) < \sigma_{\lambda_1,0}(h,B)$. 
\end{lemma}

\begin{proof}
Towards a contradiction, assume that there are $\lambda, K > 0$, a warping function $h$, and sequences $B_m \to +\infty$ and $f_m \in \RW^{1,2}(\Omega)$, such that
\begin{equation}
\label{eq:contradictionboundedev}
    \|h^{1/2} f_m\|_{\RL^2(\del \Omega)} = 1 \quad \text{and} \quad \int_{\Omega} h |\nabla f_m|^2 + \lambda(h^{-1} + B_m) f_m^2 \dV \le K.
\end{equation}
    Since $h$ is bounded and bounded away from zero, the norm
    \begin{equation}
    \label{eq:equivinner}
        \|f\|^2 = \int_{\Omega} h |\nabla f|^2 \dV + \int_{\del \Omega} h f^2 \dA
    \end{equation}
    is equivalent to the standard norm on $\RW^{1,2}(\Omega)$, so that we can read off \eqref{eq:contradictionboundedev} that $(f_m)$ is a bounded sequence in $\RW^{1,2}(\Omega)$, in particular it has a weak limit $\hat f \in \RW^{1,2}(\Omega)$. Since every term under the integral in \eqref{eq:contradictionboundedev} is positive, we have that 
    \begin{equation}
        \int_\Omega f_m^2 \dV \le \frac{K}{\lambda B_m} \to 0
    \end{equation}
    so that $f_m \to \hat f = 0$ in $\RL^2(\Omega)$. However, 
    by compactness of the trace operator, we deduce that $\|h^{1/2} \hat f\|_{\RL^2(\del \Omega)} = 1$, a contradiction. 
\end{proof}
  
\subsection{Saturated bounds via continuity}

In order to prove Theorems \ref{thm:boundedinteriorvolume}, \ref{thm:MainIntro},  and \ref{thm:onedfibre}, the following sequence of partitions of unity is useful. For $m \in \N$, introduce the domain
\begin{equation}
\label{eqdef:Omegam}
    \Omega_m := \{x \in \Omega : \dist(x,\del \Omega) \ge m^{-1} \},
\end{equation}
and let $\chi_m : \Omega \to [0,1]$ be a smooth function, taking value $1$ on $\Omega_m$ and $0$ on $\del \Omega$. For any $p \in [1,\infty)$ we have that $\chi_m \to 1$ in $\RL^p(\Omega)$ as $m \to  \infty$.

\begin{proof}[Proof of Theorem \ref{thm:boundedinteriorvolume}]
  For $r \in (0,1)$ and every $m \in \N$, consider the warping functions
  \begin{equation}
      h_m = \left(\frac{r + (1 - r)\chi_m}{\int_{\Omega}r +  (1 - r)\chi_m \dV}\right)^{1/n},
  \end{equation}
  so that $|M_{h_m}| = 1$ for all $m \in \N$, and 
  \begin{equation}
  h_m\big|_{\del \Omega} \equiv  \left(\frac{r}{\int_\Omega r + (1- r)\chi_m \dV}\right)^{1/n} \xrightarrow{m \to \infty} \left(\frac{r}{|\Omega|}\right)^{1/n}.
  \end{equation}
  By Proposition \ref{prop:continuitybis}, we have that
  \begin{equation}
     \lim_{m \to \infty} \sigma_k(M_{h_m}) = \frac{|\Omega|}{r} \sigma_k(M_1).
  \end{equation}
  Taking $r \to 0$ provides our claim.
\end{proof}

  \begin{proof}[Proof of Theorem \ref{thm:MainIntro}.] 
  By Lemma \ref{lem:doubleinequality}, we already have that 
  \begin{equation}
  \label{eq:tosaturate}
  \sup\{\sigma_k(M_h) : h\le R, \ h=1  \text{ on }   \partial \Omega\} \le R^n \sigma_k(M_R) < R^{n-2} \lambda_k\frac{|\Omega|}{|\del \Omega|}.
  \end{equation}
  We aim at proving that there is a saturating sequence for the first inequality. For any warping function $h$ taking value $1$ on $\del \Omega$, consider the sequence of functions $h_m = (1 - \chi_m) h + \chi_m R$.
  By Proposition \ref{prop:continuitybis}, we have $\sigma_k(M_{h_m}) \to R^n \sigma_k(M_R)$, so that indeed \eqref{eq:tosaturate} is saturated.

Finally, the identity \eqref{Case n>2} was proved in Section \ref{Spectrum $M_R$}, together this completes the proof of Theorem~\ref{thm:MainIntro}.
\end{proof}

Finally, the proof of Theorem \ref{thm:onedfibre} uses the homogenisation construction.

\begin{proof}[Proof of Theorem \ref{thm:onedfibre}]
Fix $R > 1$ and warping function $h \le R$ with $h \equiv 1$ on $\del \Omega$. Then, by \eqref{ineq:kcompk0},
\begin{equation}
    \sigma_k(M_h) \le \sigma_{0,k}(h) 
    \le  \min_{\dim(E) = k+1} \max_{f \in E \setminus \{0\}} R \frac{\int_\Omega |\nabla f|^2 \dV}{\int_{\del \Omega} f^2 \dA} 
    = R \sigma_{0,k}(\Omega).
\end{equation}
We now need to show that this inequality is saturated. 
Consider once again the sequence of warping functions $h_m = (1 - \chi_m) h + \chi_m R$, it follows from Proposition \ref{prop:continuitybis} that $\sigma_{\lambda_j,k}(M_{h_m}) \to R \sigma_{\lambda_j,k}(M_R)$ for all $j,k \ge 0$. 

Using Lemma \ref{lem:lowerBbound}. choose $B_m$ such that $\sigma_k(M_{h_m}) < \sigma_{\lambda_1,0}(h_m,B_m)$. Consider the sequence of warping functions $h_{m,\ell}$ provided by Proposition \ref{prop:homogenisation} so that for all $j,k \ge 0$, 
\begin{equation}
\lim_{\ell \to \infty} \sigma_{\lambda_j,k}(h_{m,\ell}) = \sigma_{\lambda_{j,k}}(h_m,B_m).
\end{equation}
In particular, for $\ell$ sufficiently large, 
\begin{equation}
\sigma_k(M_{h_{m,\ell}}) = \sigma_{0,k}(h_{m,\ell}) \xrightarrow{\ell \to \infty} \sigma_{0,k} (h_m) \xrightarrow{m \to \infty} R \sigma_k(\Omega).
\end{equation}
Therefore, we can find a diagonal sequence $h_{m,\ell(m)}$ saturating the bound
\begin{equation}
    \sup\{\sigma_k(M_h) : h \le R, \ h \equiv 1 \text{ on } \del \Omega\} = R \sigma_k(\Omega).
\end{equation}
\end{proof}

\subsection{Stability for fibres of dimension \texorpdfstring{$n=2$}{n=2}}
In this section, we first present a quantitative improvement of inequality~\eqref{ineq:introgeneral}, from which Theorem \ref{thm:Stability} follows.

\begin{proposition}\label{prop:stablebody}
    Let $D\subset\Omega$ be a subdomain and $f : \Omega \to [0,\infty)$ be a nonnegative Lipschitz function whose support is $D$. Let $h$ be a warping function such that $\int_{\partial \Omega} h^2 \, dA_{\partial \Omega} = |\partial \Omega|$. Then,
    \begin{equation}
    \label{eq:boundwithdefect}
        \sigma_k(M_h) \le \frac{\lambda_k}{|\del \Omega}| \int_\Omega h^{n-2} \dV - \frac{\lambda_k^2}{|\del \Omega}| \frac{(\int_D f h^{n-2} \dV)^2}{ \int_{D} (h^n |\nabla f|^2 + \lambda_k h^{n-2} f^2) \dV}
    \end{equation}
\end{proposition}

\begin{proof}
     Since $f$ is supported on $D$, and in particular vanishes on $\del \Omega$, we have for $t \in \R$:
    \begin{equation}
    \label{eq:rayleighconvexcomb}
    \begin{aligned}
    \CR_{\lambda_k,h}(1-tf) &= \frac{\int_{\Omega} t^2 h^n |\nabla f|^2 + \lambda_k h^{n-2} (1 - tf)^2 \dV}{\int_{\del \Omega} h^n \dA} \\
    &= \frac{\lambda_k}{|\del \Omega}| \int_\Omega h^{n-2} \dV + \frac{t^2\int_D (h^n |\nabla f|^2 + \lambda_k h^{n-2} f^2) \dV - 2 \lambda_k t \int_D h^{n-2} f \dV }{|\del \Omega|} \\
    \end{aligned}
    \end{equation}
    This expression is quadratic in $t$, its minimum is achieved when
    \begin{equation}
    t=t_0:=\frac{\lambda_k\int_D fh^{n-2}\,dV_{\Omega}}{\int_D (|df|^2h^n+\lambda_k f^2h^{n-2})\,dV_{\Omega}},
    \end{equation}
Whence $\sigma_k(M_h) \le \sigma_{\lambda_k,0}(h) \le \CR_{\lambda_k,h}(1 - t_0 f)$, so that inserting $t_0$ in \eqref{eq:rayleighconvexcomb} gives precisely \eqref{eq:boundwithdefect}.
\end{proof}    

\begin{proof}[Proof of Theorem \ref{thm:Stability}] Now the fibre $\Sigma$ has dimension $n=2$ so that $\int_{\partial \Omega} h^2 \dA = |\partial \Omega|$. Let $q \in \Omega$ and $r > 0$ be such that $B(q,r) \cap \del \Omega = \varnothing.$ Put $\delta := \tfrac{\lambda_k |\Omega|}{|\del \Omega|} - \sigma_k(M_h) > 0$. For $x \in B(q,r)$ define $f(x) = \dist(x,\del B(q,r))$ and extend it by $0$ outside. It is a direct computation to see that $|\nabla f| = 1$ almost everywhere, $f \le r$ and $f \ge r/2$ on $B(q,r/2)$. Rearranging Proposition \ref{prop:stablebody}, we read
\begin{equation}\label{ineq:withdeficit}
\delta \ge \frac{\lambda_k^2 r^2 |B(q,r/2)|^2}{4|\del \Omega| \int_{B(q,r)} h^2 + \lambda_k r^2 \dV}.
\end{equation}
Rearranging again, this reads as
\begin{equation}
    \int_{B(q,r)} h^2 \dV \ge \frac{\lambda_k^2 r^2}{4\delta} \frac{|B(q,r/2)|^2}{|\del \Omega|}  - \lambda_k r^2 |B(q,r)|.
\end{equation}
We deduce \eqref{eq:limstab} from the fact that every subdomain $D$ must contain some ball.
\end{proof}

\subsection{Unboundedness with respect to less restrictive \texorpdfstring{$\RL^p$}{Lp} norms.}

The main result of this section is Theorem \ref{thm:3dimconnected}, where the eigenvalues are unbounded in terms of an $\RL^p$ norm constraint for $1 \le p < n-2$ when $\Omega$ has connected boundary. We also show the contrasting Theorem \ref{dim basis 1}, where boundedness is preserved in a situation with two boundary components.

In order to prove Theorem \ref{thm:3dimconnected}, we recall the fact that for any warping function $h$, $\sigma_{1}(M_h) = \min\{\sigma_{0,1}(h),\sigma_{\lambda_1,0}(h)\}$. As such, we need to find a sequence of warping functions $h_m$ that is bounded in $\RL^p$ so that both of these quantities become unbounded as $m \to \infty$. 

To this end, consider $\Omega_m$ defined as in \eqref{eqdef:Omegam}, and for $R > 0$ and $r \in (0,1)$, define a sequence of smooth warping functions $h_m$ satisfying the following constraints:
\begin{equation}
    \label{eqdef:hmLp}
    \begin{cases}
    h_m = \left(\frac{Rm}{2 |\del \Omega|}\right)^{1/p} & \text{on } \Omega_{2m} \setminus \Omega_m, \\
    h_m \ge r & \text{on } \Omega \setminus \Omega_{2m}, \\
    \|h_m\|_{\RL^p(\Omega)} = R,  \\
    h_m \equiv 1 & \text{on } \del \Omega.
    \end{cases}
\end{equation}
\begin{lemma}
For $h_m$ defined as in \eqref{eqdef:hmLp}, $\sigma_{\lambda_1,0}(h_m) \to \infty$ as $m \to \infty$.
\end{lemma}

\begin{proof}
    Let $f_m$ be an eigenfunction corresponding to $\sigma_{\lambda_1,0}(h_m)$, normalised to $\|f_m\|_{\RL^2(\del \Omega)} = 1$. We consider two cases separately, recalling that $(\del \Omega)_t$ is the set at distance $t$ from the boundary.
    
    \noindent \textbf{First case:} If, $\int_{\Omega_{2m} \setminus \Omega_{m}} f_m^2 \dV \ge (2m)^{-1}$, then,  because $p < n- 2$ and $\lambda_1 > 0$, we have
    \begin{equation}
        \int_{\Omega_{2m} \setminus \Omega_{m}} \lambda_1 f_m^2 h_m^{n-2} \ge \frac{\lambda_1}{2m} h_m^{n-2} = \frac{\lambda_1}{2}\left(\frac{R}{2|\del \Omega|}\right)^{\frac{n-2}{p}} m^{\frac{n-2 - p}{p}} \to \infty
    \end{equation}
    as $m \to \infty$.

    \noindent \textbf{Second case:} In the complementary case, there has to be some $t \in ( (2m)^{-1},m^{-1})$ so that 
    $\int_{(\del \Omega)_t} f_m^2 \de A_{(\del \Omega)_t} < 1/2$. Using the normalisation $\|f_m\|_{\RL^2(\del \Omega)} = 1$ and Lemma \ref{lemma:deBrunopasproduit}, we have that
    \begin{equation}
        \frac 1 2 < \abs{1 - \int_{(\del \Omega)_t} f_m^2 \de A_{(\del \Omega)_t}} \lesssim t^{1/2}(1 + \|\nabla f_m\|_{\RL^2(\Omega \setminus \Omega_m)})^2 \lesssim m^{-1/2} (1  + \|\nabla f_m\|_{\RL^2(\Omega \setminus \Omega_m)})^2.
    \end{equation}
    In particular, this implies that
    \begin{equation}
        \|\nabla f_m\|_{\RL^2(\Omega \setminus \Omega_m)}^2 \gtrsim m^{1/2},
    \end{equation}
    so that
    \begin{equation}
        \int_{\Omega \setminus \Omega_m} h_m^n|\nabla f_m|^2\dV \gtrsim r^n m^{1/2} \xrightarrow{m \to \infty} \infty.
    \end{equation}
\end{proof}

\begin{lemma}
    For $h_m$ defined in \eqref{eqdef:hmLp}, $\sigma_{0,1}(h_m) \to \infty$ as $m \to \infty$.
\end{lemma}
\begin{proof}
    Let $f$ be a normalised eigenfunction corresponding to $\sigma_{0,1}$, which is in particular orthogonal to constants on $\del \Omega$, so that we have
    \begin{equation}
        \int_{\del \Omega} f^2 \dA = 1, \qquad \int_{\del \Omega} f \dA = 0.
    \end{equation}
    This time we consider three cases.

    \noindent \textbf{First case:} Suppose that $\int_{\Omega_{2m} \setminus \Omega_m} |\nabla f|^2 \dV \ge 1$. Then,
    \begin{equation}
        \sigma_{0,1}(h_m) \ge \int_{\Omega_{2m} \setminus \Omega_m} h_m^n |\nabla f|^2 \dV \ge \left(\frac{Rm}{2 |\del \Omega|} \right)^{n/p} \xrightarrow{m \to \infty} \infty.
    \end{equation}
\vspace{2ex}

    \noindent \textbf{Second case: } Suppose that $\int_{\Omega \setminus \Omega_{2m}} |\nabla f|^2 \dV \ge m^{1/4}$. Then, since $h_m \ge r$ on $\Omega \setminus \Omega_{2m}$, we have 
    \begin{equation}
        \sigma_{0,1}(h_m) \ge \int_{\Omega \setminus \Omega_{2m}} h_m^n |\nabla f|^2 \dV \ge r^n m^{1/4} \xrightarrow{m \to \infty} \infty.
    \end{equation}
    \vspace{2ex} 

    \noindent \textbf{Third case:} We now assume both
    \begin{equation}
         \int_{\Omega \setminus \Omega_{2m}} |\nabla f|^2 \dV \le m^{1/4} \qquad \text{and} \qquad \int_{\Omega_{2m} \setminus \Omega_m} |\nabla f|^2 \dV \le 1
    \end{equation}
    Recalling the eigenfunction normalisation in $\RL^2(\del \Omega)$, for each $t \in ( (2m)^{-1}, m^{-1})$ we have by Lemma \ref{lemma:deBrunopasproduit} that
    \begin{equation}
        \begin{aligned}
            \abs{1 - \|f\|_{\RL^2((\del \Omega)_t})}^2 \lesssim t^{1/2}(1 + \|\nabla f\|_{\RL^2(\Omega \setminus \Omega_m)})^2  \lesssim m^{-1/4}.
        \end{aligned}
    \end{equation}
    Applying the inner product bound in Lemma \ref{lemma:deBrunopasproduit} with $\tilde f = 1$ and using the fact that $f$ is orthogonal to constants on $\del \Omega$ then gives
    \begin{equation}
        \begin{aligned}
            \abs{\int_{(\del \Omega)_t} f \de A_{\del \Omega_t}} \lesssim t^{1/2} (1 + |\del \Omega| +  \|\nabla f\|_{\RL^2(\Omega \setminus \Omega_m)})^2 \lesssim m^{-1/4}.
        \end{aligned}
    \end{equation}
    As in Section \ref{sec:teclem}, denoting by $\phi_t : \del \Omega \to (\del \Omega)_t$ the diffeomorphism obtained by following the normal geodesics, we have that $\phi_t^* g_{(\del \Omega)_t} = g_{\del \Omega} + \smallo t$, so that $\lambda_1((\del \Omega)_t) = \lambda_1(\del \Omega) + \smallo 1$, and $\lambda_1(\del \Omega) > 0$ as it is connected. Thus, we use $f\big|_{(\del \Omega)_t}$ as a trial function for $\lambda_1((\del \Omega)_t)$ to see
    \begin{equation}
        \int_{(\del \Omega)_t} |\nabla f|^2 \de A_{(\del \Omega)_t} \ge (\lambda_1(\del \Omega) + \smallo 1) \int_{(\del \Omega)_t} f^2 \de A_{(\del \Omega)_t},
    \end{equation}
    which in turn gives
    \begin{equation}
\int_{\Omega_{2m} \setminus \Omega_m} |\nabla f|^2 \dV \ge (\lambda_1(\del \Omega) + \smallo 1) \int_{\Omega_{2m} \setminus \Omega_m} f^2 \dV \gtrsim m^{-1},
    \end{equation}
    from which we deduce
    \begin{equation}
    \int_{\Omega_{2m} \setminus \Omega_m} h_m^n |\nabla f|^2 \dV \gtrsim m^{\frac n p - 1} \xrightarrow{m \to \infty} \infty,
    \end{equation}
    since $p \leq n-2$, completing the proof.
  
\end{proof}

Finally, we prove that in the case where the base $\Omega$ is one-dimensional and so does not have connected boundary,
the first eigenvalue is bounded in terms of the $\RL^p$ norm of $h$ for all $p \ge 1$.

\begin{proof}[Proof of Theorem \ref{dim basis 1}.]

We put
\begin{equation}
A = \Big\{x \in [0,L] : h(x) < \Big(\frac{n}{(n-1) L}\Big)^{1/p} \|h\|_{\RL^p} \Big\}.
\end{equation}
Chebyshev's inequality tells us that $|[0,L] \setminus A| \le \frac{(n-1)L}{n}$, and in particular $A$ is an open set of length at least $\frac{L}{n}$. Consider the continuous trial function $f$ such that
\begin{equation}
    \begin{cases}
        f(0) = -1 & f(L) = 1 \\
        f'(x) = 0 & \text{for } x \in [0,L] \setminus A; \\
        f'(x) = \frac{2}{|A|} & \text{for } x \in A.
    \end{cases}
\end{equation}
It increases with slope $f'(x) = \tfrac{2}{|A|}$ on $A$, and is constant everywhere else. We can therefore bound
\begin{equation}
  \sigma_1(M_h) \le \frac 1 2 \int_0^L h^n (f')^2 \de x \le \int_A \frac{n^{n/p}\|h\|_{\RL^p}^n}{((n-1)L)^{n/p}}\frac{2}{|A|^2} \de x \le 2 n \left(\frac{n}{n-1}\right)^{\frac{n}{p}}{L^{-\frac{n+p}{p}}} \|h\|_{\RL^p}^n.
\end{equation}
We are only left with observing that
\begin{equation}
    \left(\frac{n}{n-1}\right)^{\frac{n}{p}} \le \left(\frac{n}{n-1}\right)^{n} \le 4
\end{equation}
to obtain the desired inequality
\begin{equation}
    \sigma_1(M_h) \le  \frac{8 n}{L^{\frac{n + p}{p}}}\|h\|_{\RL^p}^n.
\end{equation}

\end{proof}

% \section*{Open access statement}
% For the purpose of open access, the authors have applied a Creative Commons Attribution (CC BY) licence to any Author Accepted Manuscript version arising from this submission.

% \subsection*{Competing Interests}
% None.

% \subsection*{Data availability statement}
% The authors declare that the data supporting the findings of this study are available within the paper. We are not using external data.

\bibliographystyle{plain}              
\bibliography{bibliography}

@incollection {BucFreKen2017,
    AUTHOR = {Bucur, Dorin and Freitas, Pedro and Kennedy, James},
     TITLE = {The {R}obin problem},
 BOOKTITLE = {Shape optimization and spectral theory},
     PAGES = {78--119},
 PUBLISHER = {De Gruyter Open, Warsaw},
      YEAR = {2017},
      ISBN = {978-3-11-055088-7; 978-3-11-055085-6},
   MRCLASS = {58J32 (58J50)},
  MRNUMBER = {3681148},
MRREVIEWER = {Juan\ Manuel\ P\'erez-Pardo},
       DOI = {10.1515/9783110550887-004},
       URL = {https://doi-org.acces.bibl.ulaval.ca/10.1515/9783110550887-004},
}

@article {Lo2025,
    AUTHOR = {Lo, Chia-Chun},
     TITLE = {Homogenisation for the {R}obin eigenvalue problem on manifolds
              and flexibility of optimal {S}chr\"odinger potentials},
   JOURNAL = {J. Geom. Anal.},
  FJOURNAL = {Journal of Geometric Analysis},
    VOLUME = {35},
      YEAR = {2025},
    NUMBER = {11},
     PAGES = {Paper No. 364, 40},
      ISSN = {1050-6926,1559-002X},
   MRCLASS = {35P15 (35B27 47A75 58C40)},
  MRNUMBER = {4961407},
       DOI = {10.1007/s12220-025-02188-w},
       URL = {https://doi.org/10.1007/s12220-025-02188-w},
}

@article {GiPoly2024,
    AUTHOR = {Girouard, Alexandre and Polymerakis, Panagiotis},
     TITLE = {Large {S}teklov {E}igenvalues {U}nder {V}olume {C}onstraints},
   JOURNAL = {J. Geom. Anal.},
  FJOURNAL = {Journal of Geometric Analysis},
    VOLUME = {34},
      YEAR = {2024},
    NUMBER = {10},
     PAGES = {Paper No. 316},
      ISSN = {1050-6926,1559-002X},
   MRCLASS = {58C40 (58J05)},
  MRNUMBER = {4788000},
       DOI = {10.1007/s12220-024-01768-6},
       URL = {https://doi.org/10.1007/s12220-024-01768-6},
}

@article {CoGiGoSh2024,
    AUTHOR = {Colbois, Bruno and Girouard, Alexandre and Gordon, Carolyn and
              Sher, David},
     TITLE = {Some recent developments on the {S}teklov eigenvalue problem},
   JOURNAL = {Rev. Mat. Complut.},
  FJOURNAL = {Revista Matem\'atica Complutense},
    VOLUME = {37},
      YEAR = {2024},
    NUMBER = {1},
     PAGES = {1--161},
      ISSN = {1139-1138,1988-2807},
   MRCLASS = {58C40 (35P05 35P15 35P20 53A10 58J50)},
  MRNUMBER = {4695859},
       DOI = {10.1007/s13163-023-00480-3},
       URL = {https://doi.org/10.1007/s13163-023-00480-3},
}

@book {LeMaPo2023,
    AUTHOR = {Levitin, Michael and Mangoubi, Dan and Polterovich, Iosif},
     TITLE = {Topics in spectral geometry},
    SERIES = {Graduate Studies in Mathematics},
    VOLUME = {237},
 PUBLISHER = {American Mathematical Society, Providence, RI},
      YEAR = {[2023] \copyright2023},
     PAGES = {xviii+325},
      ISBN = {[9781470475253]; [9781470475482]; [9781470475499]},
   MRCLASS = {35-01 (35Pxx 47A75 58J50 58J53 65N30)},
  MRNUMBER = {4655924},
       DOI = {10.1090/gsm/237},
       URL = {https://doi.org/10.1090/gsm/237},
}

@article{BrCoGi2024, title={Spectral ratios and gaps for {S}teklov eigenvalues of balls with revolution-type metrics}, volume={68}, DOI={10.4153/S0008439524000778}, number={2}, journal={Canadian Mathematical Bulletin}, author={Brisson, Jade and Colbois, Bruno and Gittins, Katie}, year={2025}, pages={492–511}}

@article {RaTa1975,
    AUTHOR = {Rauch, Jeffrey and Taylor, Michael},
     TITLE = {Potential and scattering theory on wildly perturbed domains},
   JOURNAL = {J. Functional Analysis},
  FJOURNAL = {Journal of Functional Analysis},
    VOLUME = {18},
      YEAR = {1975},
     PAGES = {27--59},
      ISSN = {0022-1236},
   MRCLASS = {35P25 (47A40)},
  MRNUMBER = {377303},
MRREVIEWER = {Jean-Claude Guillot},
       DOI = {10.1016/0022-1236(75)90028-2},
       URL = {https://doi-org.acces.bibl.ulaval.ca/10.1016/0022-1236(75)90028-2},
}

@article {KaLa2022,
    AUTHOR = {Karpukhin, Mikhail and Lagac\'e, Jean},
     TITLE = {Flexibility of {S}teklov eigenvalues via boundary
              homogenisation},
   JOURNAL = {Ann. Math. Qu\'e.},
  FJOURNAL = {Annales Math\'ematiques du Qu\'ebec},
    VOLUME = {48},
      YEAR = {2024},
    NUMBER = {1},
     PAGES = {175--186},
      ISSN = {2195-4755,2195-4763},
   MRCLASS = {58J50 (35B27 35J20 35P15)},
  MRNUMBER = {4718574},
MRREVIEWER = {Feng\ Du},
       DOI = {10.1007/s40316-022-00207-8},
       URL = {https://doi-org.acces.bibl.ulaval.ca/10.1007/s40316-022-00207-8},
}

@article {Xi2022,
    AUTHOR = {Xiong, Changwei},
     TITLE = {On the spectra of three {S}teklov eigenvalue problems on
              warped product manifolds},
   JOURNAL = {J. Geom. Anal.},
  FJOURNAL = {Journal of Geometric Analysis},
    VOLUME = {32},
      YEAR = {2022},
    NUMBER = {5},
     PAGES = {Paper No. 153, 35},
      ISSN = {1050-6926},
   MRCLASS = {58C40 (35P15)},
  MRNUMBER = {4386421},
       DOI = {10.1007/s12220-022-00889-0},
       URL = {https://doi.org/10.1007/s12220-022-00889-0},
}

@article {CiGi2018,
    AUTHOR = {Cianci, Donato and Girouard, Alexandre},
     TITLE = {Large spectral gaps for {S}teklov eigenvalues under volume
              constraints and under localized conformal deformations},
   JOURNAL = {Ann. Global Anal. Geom.},
  FJOURNAL = {Annals of Global Analysis and Geometry},
    VOLUME = {54},
      YEAR = {2018},
    NUMBER = {4},
     PAGES = {529--539},
      ISSN = {0232-704X},
   MRCLASS = {58J50 (53C20 53C21)},
  MRNUMBER = {3878841},
MRREVIEWER = {Ahmed Lesfari},
       DOI = {10.1007/s10455-018-9612-6},
       URL = {https://doi.org/10.1007/s10455-018-9612-6},
}

@article {CoElGi2019,
    AUTHOR = {Colbois, Bruno and El Soufi, Ahmad and Girouard, Alexandre},
     TITLE = {Compact manifolds with fixed boundary and large {S}teklov
              eigenvalues},
   JOURNAL = {Proc. Amer. Math. Soc.},
  FJOURNAL = {Proceedings of the American Mathematical Society},
    VOLUME = {147},
      YEAR = {2019},
    NUMBER = {9},
     PAGES = {3813--3827},
      ISSN = {0002-9939},
   MRCLASS = {58J50 (35P15 53C20 53C21 53C23)},
  MRNUMBER = {3993774},
MRREVIEWER = {Julian Edward},
       DOI = {10.1090/proc/14426},
       URL = {https://doi.org/10.1090/proc/14426},
}

@article {FrSc2016,
    AUTHOR = {Fraser, Ailana and Schoen, Richard},
     TITLE = {Sharp eigenvalue bounds and minimal surfaces in the ball},
   JOURNAL = {Invent. Math.},
  FJOURNAL = {Inventiones Mathematicae},
    VOLUME = {203},
      YEAR = {2016},
    NUMBER = {3},
     PAGES = {823--890},
      ISSN = {0020-9910},
   MRCLASS = {58C40 (35P15 53A10)},
  MRNUMBER = {3461367},
MRREVIEWER = {Isabel M. C. Salavessa},
       DOI = {10.1007/s00222-015-0604-x},
       URL = {https://doi-org.acces.bibl.ulaval.ca/10.1007/s00222-015-0604-x},
}

@article {CoGiGi2019,
    AUTHOR = {Colbois, Bruno and Girouard, Alexandre and Gittins, Katie},
     TITLE = {Steklov eigenvalues of submanifolds with prescribed boundary
              in {E}uclidean space},
   JOURNAL = {J. Geom. Anal.},
  FJOURNAL = {Journal of Geometric Analysis},
    VOLUME = {29},
      YEAR = {2019},
    NUMBER = {2},
     PAGES = {1811--1834},
      ISSN = {1050-6926},
   MRCLASS = {58C40 (35P15 35R01 58J50)},
  MRNUMBER = {3935280},
MRREVIEWER = {Julian Edward},
       DOI = {10.1007/s12220-018-0063-x},
       URL = {https://doi.org/10.1007/s12220-018-0063-x},
}

@article {CoVe2021,
    AUTHOR = {Colbois, Bruno and Verma, Sheela},
     TITLE = {Sharp {S}teklov upper bound for submanifolds of revolution},
   JOURNAL = {J. Geom. Anal.},
  FJOURNAL = {Journal of Geometric Analysis},
    VOLUME = {31},
      YEAR = {2021},
    NUMBER = {11},
     PAGES = {11214--11225},
      ISSN = {1050-6926},
   MRCLASS = {58C40 (35P15)},
  MRNUMBER = {4310168},
       DOI = {10.1007/s12220-021-00678-1},
       URL = {https://doi.org/10.1007/s12220-021-00678-1},
}

@article {DaKaNi2021,
    AUTHOR = {Daud\'{e}, Thierry and Kamran, Niky and Nicoleau, Fran\c{c}ois},
     TITLE = {Stability in the inverse {S}teklov problem on warped product
              {R}iemannian manifolds},
   JOURNAL = {J. Geom. Anal.},
  FJOURNAL = {Journal of Geometric Analysis},
    VOLUME = {31},
      YEAR = {2021},
    NUMBER = {2},
     PAGES = {1821--1854},
      ISSN = {1050-6926},
   MRCLASS = {35P25 (58J50 81U40)},
  MRNUMBER = {4215278},
       DOI = {10.1007/s12220-019-00326-9},
       URL = {https://doi.org/10.1007/s12220-019-00326-9},
}

@article {DaHeNi2021,
    AUTHOR = {Daud\'e, Thierry and Helffer, Bernard and Nicoleau, Fran\c{c}ois},
     TITLE = {Exponential localization of {S}teklov eigenfunctions on warped
              product manifolds: the flea on the elephant phenomenon},
   JOURNAL = {Ann. Math. Qu\'e.},
  FJOURNAL = {Annales Math\'ematiques du Qu\'ebec},
    VOLUME = {47},
      YEAR = {2023},
    NUMBER = {2},
     PAGES = {295--330},
      ISSN = {2195-4755,2195-4763},
   MRCLASS = {35P15 (58J50)},
  MRNUMBER = {4645693},
MRREVIEWER = {Ayman\ Kachmar},
       DOI = {10.1007/s40316-021-00185-3},
       URL = {https://doi.org/10.1007/s40316-021-00185-3},
}

@article {GiPo2017,
    AUTHOR = {Girouard, Alexandre and Polterovich, Iosif},
     TITLE = {Spectral geometry of the {S}teklov problem (survey article)},
   JOURNAL = {J. Spectr. Theory},
  FJOURNAL = {Journal of Spectral Theory},
    VOLUME = {7},
      YEAR = {2017},
    NUMBER = {2},
     PAGES = {321--359},
      ISSN = {1664-039X},
   MRCLASS = {58J50 (35J25 35R01)},
  MRNUMBER = {3662010},
       URL = {https://doi.org/10.4171/JST/164},
}

@article {GiHeLa2021,
    AUTHOR = {Girouard, Alexandre and Henrot, Antoine and Lagac\'{e}, Jean},
     TITLE = {From {S}teklov to {N}eumann via homogenisation},
   JOURNAL = {Arch. Ration. Mech. Anal.},
  FJOURNAL = {Archive for Rational Mechanics and Analysis},
    VOLUME = {239},
      YEAR = {2021},
    NUMBER = {2},
     PAGES = {981--1023},
      ISSN = {0003-9527},
   MRCLASS = {35J15 (47A75)},
  MRNUMBER = {4201620},
MRREVIEWER = {Zakaria El Allali},
       DOI = {10.1007/s00205-020-01588-2},
       URL = {https://doi-org.acces.bibl.ulaval.ca/10.1007/s00205-020-01588-2},
}

@article {GiLa2021,
    AUTHOR = {Girouard, Alexandre and Lagac\'{e}, Jean},
     TITLE = {Large {S}teklov eigenvalues via homogenisation on manifolds},
   JOURNAL = {Invent. Math.},
  FJOURNAL = {Inventiones Mathematicae},
    VOLUME = {226},
      YEAR = {2021},
    NUMBER = {3},
     PAGES = {1011--1056},
      ISSN = {0020-9910},
   MRCLASS = {58C40 (35P30)},
  MRNUMBER = {4337975},
       DOI = {10.1007/s00222-021-01058-w},
       URL = {https://doi-org.acces.bibl.ulaval.ca/10.1007/s00222-021-01058-w},
}

@article {GiKaLa2021,
    AUTHOR = {Girouard, Alexandre and Karpukhin, Mikhail and Lagac\'{e}, Jean},
     TITLE = {Continuity of eigenvalues and shape optimisation for {L}aplace
              and {S}teklov problems},
   JOURNAL = {Geom. Funct. Anal.},
  FJOURNAL = {Geometric and Functional Analysis},
    VOLUME = {31},
      YEAR = {2021},
    NUMBER = {3},
     PAGES = {513--561},
      ISSN = {1016-443X},
   MRCLASS = {49R05 (35P05 46E35 49Q10)},
  MRNUMBER = {4311579},
       DOI = {10.1007/s00039-021-00573-5},
       URL = {https://doi-org.acces.bibl.ulaval.ca/10.1007/s00039-021-00573-5},
}

@article {FrSc2011,
    AUTHOR = {Fraser, Ailana and Schoen, Richard},
     TITLE = {The first {S}teklov eigenvalue, conformal geometry, and
              minimal surfaces},
   JOURNAL = {Adv. Math.},
  FJOURNAL = {Advances in Mathematics},
    VOLUME = {226},
      YEAR = {2011},
    NUMBER = {5},
     PAGES = {4011--4030},
      ISSN = {0001-8708},
   MRCLASS = {58J32 (53A10 58J50)},
  MRNUMBER = {2770439},
MRREVIEWER = {Ahmad El Soufi},
       DOI = {10.1016/j.aim.2010.11.007},
       URL = {https://doi.org/10.1016/j.aim.2010.11.007},
}

@article {FaTaYu2015,
    AUTHOR = {Fan, Xu-Qian and Tam, Luen-Fai and Yu, Chengjie},
     TITLE = {Extremal problems for {S}teklov eigenvalues on annuli},
   JOURNAL = {Calc. Var. Partial Differential Equations},
  FJOURNAL = {Calculus of Variations and Partial Differential Equations},
    VOLUME = {54},
      YEAR = {2015},
    NUMBER = {1},
     PAGES = {1043--1059},
      ISSN = {0944-2669},
   MRCLASS = {58C40 (35P15 53A10)},
  MRNUMBER = {3385191},
MRREVIEWER = {Asma Hassannezhad},
       DOI = {10.1007/s00526-014-0816-8},
       URL = {https://doi.org/10.1007/s00526-014-0816-8},
}

@article {Xi2021,
    AUTHOR = {Xiong, Changwei},
     TITLE = {Optimal estimates for {S}teklov eigenvalue gaps and ratios on
              warped product manifolds},
   JOURNAL = {Int. Math. Res. Not. IMRN},
  FJOURNAL = {International Mathematics Research Notices. IMRN},
      YEAR = {2021},
      VOLUME = {2021},
    NUMBER = {22},
     PAGES = {16938--16962},
      ISSN = {1073-7928},
   MRCLASS = {58C40 (35R30)},
  MRNUMBER = {4345819},
       DOI = {10.1093/imrn/rnz258},
       URL = {https://doi.org/10.1093/imrn/rnz258},
}

@article {Tsc24,
    AUTHOR = {Tschanz, L\'eonard},
     TITLE = {Sharp upper bounds for {S}teklov eigenvalues of a hypersurface
              of revolution with two boundary components in {E}uclidean
              space},
   JOURNAL = {Ann. Math. Qu\'e.},
  FJOURNAL = {Annales Math\'ematiques du Qu\'ebec},
    VOLUME = {48},
      YEAR = {2024},
    NUMBER = {2},
     PAGES = {489--518},
      ISSN = {2195-4755,2195-4763},
   MRCLASS = {58J50},
  MRNUMBER = {4798686},
MRREVIEWER = {Panagiotis\ Polymerakis},
       DOI = {10.1007/s40316-024-00225-8},
       URL = {https://doi.org/10.1007/s40316-024-00225-8},
}

@article {Sel25,
    AUTHOR = {Seliutskii, Denis},
     TITLE = {Upper {B}ounds for {S}teklov {E}igenvalues of a {H}ypersurface
              of {R}evolution},
   JOURNAL = {Russ. J. Math. Phys.},
  FJOURNAL = {Russian Journal of Mathematical Physics},
    VOLUME = {32},
      YEAR = {2025},
    NUMBER = {2},
     PAGES = {399--407},
      ISSN = {1061-9208,1555-6638},
   MRCLASS = {35P15},
  MRNUMBER = {4939345},
       DOI = {10.1134/S1061920824601605},
       URL = {https://doi.org/10.1134/S1061920824601605},
}

@article {Ej79,
    AUTHOR = {Ejiri, Norio},
     TITLE = {{A} construction of non-flat, compact irreducible {R}iemannian manifolds which are isospectral but not isometric},
   JOURNAL = {Math. Z.},
  FJOURNAL = {Mathematische Zeitschrift},
    VOLUME = {168},
      YEAR = {1979},
    NUMBER = {2},
     PAGES = {207--212},
      ISSN = {},
   MRCLASS = {},
  MRNUMBER = {},
       DOI = {10.1007/BF01214512},
       URL = {https://doi.org/10.1007/BF01214512},
}

@article {Tsu80,
    AUTHOR = {Tsukada, Kazumi},
     TITLE = {Eigenvalues of the {L}aplacian of warped product},
   JOURNAL = {Tokyo J. Math.},
  FJOURNAL = {Tokyo Journal of Mathematics},
    VOLUME = {3},
      YEAR = {1980},
    NUMBER = {1},
     PAGES = {131--136},
      ISSN = {0387-3870},
   MRCLASS = {58G25 (53C20)},
  MRNUMBER = {584550},
MRREVIEWER = {Th.\ Friedrich},
       DOI = {10.3836/tjm/1270216086},
       URL = {https://doi.org/10.3836/tjm/1270216086},
}
\end{document}